\documentclass[hidelinks,onefignum,onetabnum]{siamart251216}

\ifpdf
\hypersetup{
  pdftitle={Broken-space additive Schwarz approximations to conforming finite-element mass inverses},
  pdfauthor={O. A. Krzysik, B. S. Southworth, and G. A. Wimmer}
}
\fi

\usepackage{amsmath,amssymb,amsfonts}
\usepackage{mathtools}
\mathtoolsset{centercolon}

\usepackage{bm}
\usepackage{xcolor}
\usepackage{graphicx}
\usepackage{epstopdf}
\usepackage{algorithmic}
\usepackage{tikz}
\usetikzlibrary{shapes.geometric, calc}
\usepackage{subcaption}
\usepackage{array}
\usepackage{booktabs}
\usepackage{comment}
\usepackage[normalem]{ulem}

\ifpdf
  \DeclareGraphicsExtensions{.pdf,.png,.jpg,.eps}
\else
  \DeclareGraphicsExtensions{.eps}
\fi

\headers{Broken-space mass inverse approximations}
{O. A. Krzysik, B. S. Southworth, and G. A. Wimmer}

\title{Broken-space Additive Schwarz Mass Inverse Approximations and (Block) Preconditioning
\thanks{\funding{This work was supported by the Laboratory Directed Research and Development program of Los Alamos National Laboratory, under project number 20240261ER. Los Alamos National Laboratory report number LA-UR-26-25735.}}}

\author{
Oliver A. Krzysik\thanks{Theoretical Division, Los Alamos National Laboratory
(\email{okrzysik@lanl.gov}, \url{https://orcid.org/0000-0001-7880-6512}.}
\and
Ben S. Southworth\thanks{Theoretical Division, Los Alamos National Laboratory
(\email{southworth@lanl.gov}, \url{https://orcid.org/0000-0002-0283-4928}.}
\and
Golo A. Wimmer\thanks{Theoretical Division, Los Alamos National Laboratory
(\email{gwimmer@lanl.gov}, \url{https://orcid.org/0000-0002-7871-1748}.}
}

\newsiamthm{remark}{Remark}
\newsiamremark{assumption}{Assumption}
\newsiamremark{hypothesis}{Hypothesis}
\newsiamthm{claim}{Claim}

\crefname{assumption}{Assumption}{Assumptions}
\Crefname{assumption}{Assumption}{Assumptions}
\crefname{hypothesis}{Hypothesis}{Hypotheses}
\Crefname{hypothesis}{Hypothesis}{Hypotheses}
\crefname{claim}{Claim}{Claims}
\Crefname{claim}{Claim}{Claims}

\crefname{section}{Section}{Sections}
\Crefname{section}{Section}{Sections}
\crefname{subsection}{Section}{Sections}
\Crefname{subsection}{Section}{Sections}

\newcolumntype{P}[1]{>{\raggedright\arraybackslash}p{#1}}

\allowdisplaybreaks

\newcommand{\R}{\mathbb R}
\newcommand{\Vh}{V_h}
\newcommand{\Vhbr}{V_h^{\rm br}}
\newcommand{\Qh}{Q_h}

\newcommand{\Hdiv}{H(\operatorname{div})}

\newcommand{\Th}{\mathcal T_h}
\newcommand{\ud}{\,\mathrm{d}}

\newcommand{\wh}[1]{\widehat{#1}}
\newcommand{\wt}[1]{\widetilde{#1}}

\newcommand{\Mbr}{M_{\rm br}}
\newcommand{\Mwtinv}{\wt M^{-1}}
\newcommand{\Mdiag}{M_{\rm diag}}
\newcommand{\Mbras}{M_{\rm bras}}
\newcommand{\range}{\operatorname{range}}
\newcommand{\kerop}{\operatorname{ker}}

\DeclareMathOperator{\diag}{diag}

\DeclareMathOperator{\curl}{curl}
\DeclareMathOperator{\grad}{grad}

\let\div\relax
\DeclareMathOperator{\div}{div}

\begin{document}
\maketitle

\begin{abstract}
Finite-element mass matrix solves and approximate inverses arise often in explicit time integration as well as Schur-complement-based block preconditioning.
A diagonal approximation of the mass matrix is cheap and widely used, but can be a poor approximation, particularly for high-order elements. 
This paper introduces a broken-space additive Schwarz (BRAS) mass inverse approximation, formed by applying exact element-local inverse mass matrices
on the broken finite-element space and averaging the result back to the conforming
space.  
The construction uses the same element matrices and local-to-global maps as
standard mass assembly, has the same element-adjacency sparsity graph as the
conforming mass matrix, and is symmetric positive definite for any conforming
space, mesh geometry, and polynomial basis.  We prove spectral bounds for the
preconditioned mass matrix, and wide-ranging numerical experiments for
\(H^1\), \(H(\operatorname{curl})\), and \(H(\operatorname{div})\) finite elements on two- and three-dimensional simplicial meshes show that BRAS reduces spectral condition numbers, Krylov iterations, and solve times relative to diagonal preconditioning. For preconditioned conjugate gradient (CG), BRAS yields a 1.1--4.7$\times$ speedup over diagonal/Jacobi preconditioning across all cases tested over finite-element orders $p\in[1,4]$. Further, theory and numerical experiments show that in the block-preconditioning case, BRAS can improve Schur-complement approximations and reduce outer solve times. On mixed Poisson and biharmonic systems, BRAS yields a 1.5--3$\times$ speedup in time-to-solution over a standard diagonal-based preconditioning approach. 
\end{abstract}



\tableofcontents

\section{Introduction}
\label{sec:introduction}

Let \(\Vh\) be a conforming finite-element space with basis \(\{\phi_i\}\).
The mass matrix
\(
    M_{ij} = (\phi_j,\phi_i)_\Omega
\)
is the coefficient representation of the \(L^2(\Omega)\) Riesz map on \(\Vh\). Inverting or approximately inverting mass matrices arises often in finite-element simulation of partial differential equations (PDEs), both in explicit time integration of dynamic PDEs and in block preconditioning for mixed systems of PDEs. For conforming finite-element spaces, the mass inverse is typically dense, and although often considered one of the least interesting aspects of finite-element simulations, it can have a significant effect on total computational cost and wallclock time. For explicit integration, applying the mass inverse can be a dominant computational cost, and the only globally coupled component across all spatial cells and MPI cores. In block preconditioning, the local/sparse approximation of a mass inverse within an approximate Schur complement can determine the quality of the Schur complement approximation. This in turn largely determines the convergence of preconditioned fixed-point and Krylov iterations; see, e.g., \cite{notay2014new,southworth2020fixed}. 

One line of research avoids the need for mass inverse approximations or preconditioners by directly using lumped (diagonal) mass matrices in the discretization. Analyzing such an approach becomes a question of quadrature and functional approximation \cite{raviart1973use,chen1985lumped} rather than algebraic and spectral properties of the assembled matrix. With sufficient assumptions on mesh and finite-element space, mass lumping can work well, though it is not always viable, with  higher-order discretizations \cite{chin1999higher,jund2007arbitrary}, hyperbolic equations \cite{guermond2013correction}, Sobolev spaces other than $H^1$ \cite{cohen1998gauss,egger2020second}, and non-tensor-product or curvilinear meshes taking extra care. Recent work has also considered deferred correction approaches to avoid mass solves \cite{abgrall2017high}.

Relative to solvers for differential operators, little attention has been given to fast, general-purpose solvers for mass matrices; see related commentary in \cite{ainsworth2019triangles}. In fact, diagonal (Jacobi) approximations are commonplace in practice, including in many state-of-the-art codes. 
There are likely several reasons for this.
Firstly, diagonal approximations are simple and cheap, and in many cases they are simply ``good enough,'' with many conforming elements enjoying a reasonable spectral equivalence between the mass and its diagonal, independent of mesh resolution. 
Moreover, a diagonal approximation is natural in the block-preconditioning setting, where maintaining a sparse approximate Schur complement is typically important; in fact, a diagonal inverse in this setting can perform better than its standalone mass-preconditioning behavior would suggest.\footnote{Consider the Schur complement approximation
\(S=-GM^{-1}G^{\top} \approx -G \wt M^{-1} G^\top = \wt S\), based on the mass approximation \(\wt M^{-1} \approx M^{-1}\), and notice that the mass-inverse approximation is tested only on
\(\operatorname{range}(G^{\top})\).  Thus, directions where diagonal scaling is a poor approximation to \(M^{-1}\) may be invisible to the Schur complement.  
This restricted-space effect is quantified in \Cref{subsec:schur_theory} and arises in the mixed Poisson numerical results of \Cref{sec:mixed_poisson}.}
Lastly, mass matrices typically do not have a very nuanced structure, and this likely inhibits more effective approximations, like multilevel solvers. 

Although mass preconditioning literature is relatively small, existing work gives several important points of comparison.
Wathen \cite{wathen1987realistic} showed that continuous Galerkin mass matrices are spectrally equivalent to their diagonal, independent of mesh resolution, with spectral bounds built from unassembled element matrices. This relation to element matrices leads to independence of mesh resolution, but \emph{not} finite-element order \cite{maitre1996condition}. Recent work has focused on low-order-refined algorithms that precondition high-order conforming finite-element discretizations with appropriately defined low-order discretizations \cite{dohrmann2021spectral,pazner2023lor}. Such an approach has demonstrated strong results, but requires a more involved implementation, especially if one wanted to use it in the context of block preconditioning. Specialized \(p\)-robust mass preconditioners for scalar \(H^1\)-conforming spaces have been developed that rely on carefully constructed high-order bases \cite{ainsworth2019triangles,ainsworth2021tetrahedra,ainsworth2022systematic}, and fast factorization techniques have also been developed for high-order Bernstein local element matrices \cite{allen2020structured}. 
A special interpolation--histopolation basis is also used in \cite{pazner2024hdiv}, wherein high-order Raviart--Thomas mass matrices are then spectrally equivalent to their diagonal, independent of polynomial degree, and this approximation is used to construct Schur complement approximations in solving \(\Hdiv\) saddle-point problems.

In this paper we show that there is significant room for improvement over diagonal mass inverse approximations, particularly for higher-order discretizations and for approximation of Sobolev spaces other than $H^1$. 
In contrast to the specialized and basis-function oriented approaches mentioned above, here we propose a largely algebraic, and implementationally basis-function agnostic, framework for sparse approximation of mass inverses that applies to any conforming finite-element space, and is naturally suited for both efficient explicit time integration and approximate Schur complements. Our approach is based on a sparse approximate inverse formed by applying exact element-local inverse mass matrices on the underlying broken or unassembled finite-element space and averaging the result back to the conforming space. The resulting BRoken-space Additive Schwarz (BRAS) algorithm yields a mass inverse approximation in the sparsity pattern of the original mass matrix, and provides a significantly better approximation than the matrix diagonal or scalings thereof. Its construction only requires the same element-level information needed to assemble the mass matrix, namely the element mass matrices and the local-to-global DOF map. 

The remainder of this paper is structured as follows.
\Cref{sec:bras_construction} introduces the BRAS preconditioner and provides
supporting theory for it.
\Cref{sec:mass_numerics} studies BRAS for standalone mass systems
\(M\bm u=\bm f\), reporting condition numbers, solver iteration counts, and solve times.
Despite the BRAS inverse being more expensive to apply than a diagonal inverse, in all cases tested the reduced iteration counts provided by BRAS offset its higher apply cost, yielding 1.5--6$\times$ speedups in preconditioned conjugate gradient (CG) solve time across \(H^1\), \(H(\operatorname{curl})\), and \(H(\operatorname{div})\) finite elements on two- and three-dimensional simplicial meshes, and finite element orders $p\in[1,4]$.
Aspects of using approximate mass inverses in block preconditioning are considered
in \cref{sec:block_preconditioning}, with accompanying numerical results given in
\cref{sec:block_preconditioning_numerics}. For mixed Poisson and biharmonic problems, block preconditioning based on Schur complements built from BRAS inverses yield 1.5--3$\times$ speedups over those using standard diagonal mass inverse approximations. Concluding remarks are provided in
\cref{sec:conclusions}.

\section{BRAS: Broken-space additive Schwarz mass inverse}
\label{sec:bras_construction}

\subsection{Conforming mass matrix inverse}
\label{sec:mass_matrices}

Let \(\Omega\subset\R^d\) be a bounded domain and let \(\Th\) be a finite-element mesh. Let \(\Vh\) be a conforming finite-element space, scalar- or vector-valued, with
global conforming basis \(\{\phi_i\}_{i=1}^n\).  The corresponding mass matrix is the
\(L^2(\Omega)\) Gram matrix
\begin{equation}
\label{eq:mass_matrix_def}
    M_{ij}
    =
    (\phi_j,\phi_i)_{L^2(\Omega)}
    =
    \int_\Omega \phi_j(x)\cdot\phi_i(x)\,\ud x,
    \qquad i,j=1,\ldots,n.
\end{equation}
For scalar-valued spaces, the dot product denotes ordinary multiplication; for
vector-valued spaces, it denotes the Euclidean dot product of the vector-valued basis
functions. The matrix \(M\) represents the discrete \(L^2\) Riesz map.  If \(u_h=\sum_j (\bm{u})_j\phi_j\), then the coefficient vector \(M\bm{u}\) represents the functional
\(
    v_h \mapsto (u_h,v_h)_\Omega.
\)
Thus \(M^{-1}\) maps a coefficient representation of a discrete \(L^2\) functional back to the coefficient vector of the corresponding finite-element function. The inverse is global even though the bilinear form itself is elementwise local.

Given a symmetric positive definite (SPD) inverse approximation \(\Mwtinv\approx M^{-1}\), the implied approximate mass matrix is \({\wt M}\), and the natural measure of approximation is the generalized eigenproblem
\(
    M\bm{v} = \lambda {\wt M}\bm{v}.
\)
If the generalized eigenvalues $\{\lambda\}$ are contained in a small interval near unity, this implies rapid convergence of fixed-point or preconditioned CG applied to \(M\bm{u}=\bm{f}\), and also controls the conditioning of Schur complement approximations involving \(\Mwtinv\), as discussed in \cref{sec:block_preconditioning,sec:block_preconditioning_numerics}. 
The baseline approximation for a mass preconditioner and approximate inverse is the diagonal inverse $\Mdiag^{-1} = \diag(M)^{-1}$ or slight modifications or weightings thereof. This is the cheapest SPD approximation of $M^{-1}$, but ignores all off-diagonal coupling, which is an increasing fraction of the connections in \(M\) for increasing \(p\).

\subsection{Broken-space mass inverse}

Let \(\Vh\) be a conforming finite-element space.  Define the broken companion space
\begin{equation*}
    \Vhbr = \bigoplus_{K\in\Th} V(K),
\end{equation*}
where \(V(K)\) is the local finite-element space on the element \(K\).  Let \(\Mbr\) be the mass matrix on \(\Vhbr\).  Since \(\Vhbr\) is broken over the mesh, \(\Mbr\) is block diagonal:
\begin{equation*}
    \Mbr = \operatorname{diag}_{K\in\Th} M_K,
\end{equation*}
where \(M_K\) is the local element mass matrix. The conforming mass matrix is then assembled from broken pieces. Let \(n=\dim\Vh\) and \(n_{\rm br}=\dim\Vhbr\).  Each broken DOF is associated with a unique conforming DOF; denote this map by
\(\pi:\{1,\ldots,n_{\rm br}\}\to\{1,\ldots,n\}\).  Let
\begin{equation*}
    P:\Vh\to\Vhbr
\end{equation*}
be the conforming-to-broken injection defined in coefficient form by
\begin{equation*}
    (P\bm{u})_j = \bm{u}_{\pi(j)},
    \qquad j=1,\ldots,n_{\rm br}.
\end{equation*}
Equivalently, \(P_{ji}=1\) if \(i=\pi(j)\), and \(P_{ji}=0\) otherwise.  Thus
\(P\) represents a conforming function in the broken basis by copying each
conforming DOF to all of its element-local broken copies. An illustrative example is shown in \cref{fig:P_example}. 
Then the conforming mass matrix is assembled as 
\begin{equation}
\label{eq:M_PMbrP}
    M = P^{\top}\Mbr P = \sum_{K\in\Th} P_K^{\top}M_KP_K,
\end{equation}
where \(P_K\) denotes the local interpolation to the element \(K\).

\pgfmathsetmacro{\titlegap}{0.75}
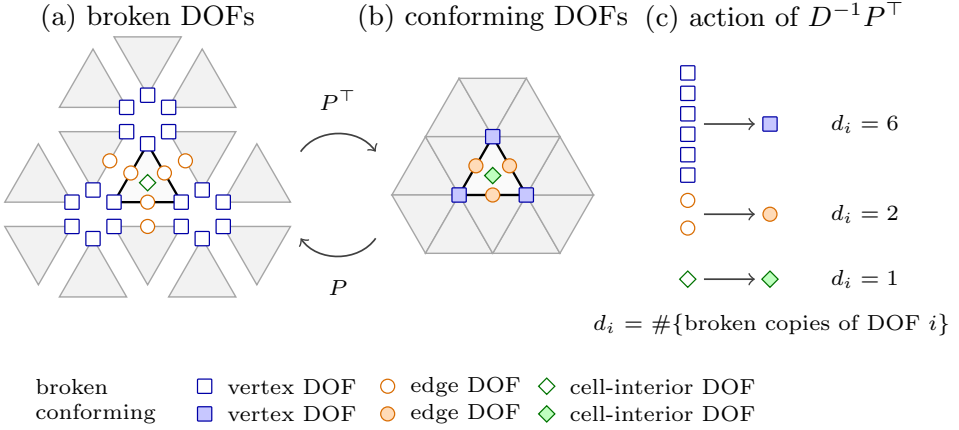
\begin{figure}[h!]
\centering
\resizebox{\linewidth}{!}{%

\pgfmathsetmacro{\titlegap}{0.75}
\begin{tikzpicture}[
    scale=0.68,
    every node/.style={font=\scriptsize},
    neightri/.style={draw=black!35, fill=black!5, line width=0.45pt},
    centtri/.style={draw=black, fill=black!2, line width=0.75pt},
    brokenvertexdof/.style={rectangle, rounded corners=0.3pt, minimum size=4.6pt,
        inner sep=0pt, draw=blue!65!black, fill=white},
    brokenedgedof/.style={circle, minimum size=4.5pt, inner sep=0pt,
        draw=orange!85!black, fill=white},
    brokencelldof/.style={diamond, aspect=1, minimum size=5.4pt, inner sep=0pt,
        draw=green!45!black, fill=white},
    confvertexdof/.style={rectangle, rounded corners=0.3pt, minimum size=4.6pt,
        inner sep=0pt, draw=blue!65!black, fill=blue!22},
    confedgedof/.style={circle, minimum size=4.5pt, inner sep=0pt,
        draw=orange!85!black, fill=orange!28},
    confcelldof/.style={diamond, aspect=1, minimum size=5.4pt, inner sep=0pt,
        draw=green!45!black, fill=green!25},
    copyarrow/.style={->, line width=0.55pt, draw=black!75},
    paneltitle/.style={font=\small}
]

\newcommand{\leftPanelDrop}{0.2cm} %

\pgfmathsetmacro{\panelTitleY}{2.95}
\pgfmathsetmacro{\patchL}{1.40}
\pgfmathsetmacro{\patchH}{0.8660254038*\patchL}
\pgfmathsetmacro{\brokenElemScale}{0.8} %
\pgfmathsetmacro{\brokenCentroidScale}{1.3} %

\pgfmathsetmacro{\confPatchScale}{0.8} %

\def\DrawBrokenTriangle#1#2#3#4#5{%
    \coordinate (#5g) at ($1/3*(#2)+1/3*(#3)+1/3*(#4)$);%
    \coordinate (#5s) at ($(lPatchCenter)+\brokenCentroidScale*(#5g)-\brokenCentroidScale*(lPatchCenter)$);%
    \coordinate (#5a) at ($(#5s)+\brokenElemScale*(#2)-\brokenElemScale*(#5g)$);%
    \coordinate (#5b) at ($(#5s)+\brokenElemScale*(#3)-\brokenElemScale*(#5g)$);%
    \coordinate (#5c) at ($(#5s)+\brokenElemScale*(#4)-\brokenElemScale*(#5g)$);%
    \draw[#1] (#5a)--(#5b)--(#5c)--cycle;%
}

\begin{scope}[xshift=0cm]
    \node[paneltitle] at (0,\panelTitleY) {(a) broken DOFs};

    \begin{scope}[yshift=-\leftPanelDrop]

    \coordinate (lA) at (0,\patchH);

    \coordinate (lA) at (0,\patchH);
    \coordinate (lB) at ({-0.5*\patchL},0);
    \coordinate (lC) at ({0.5*\patchL},0);
    \coordinate (lD) at (-\patchL,\patchH);
    \coordinate (lE) at (\patchL,\patchH);
    \coordinate (lF) at (0,-\patchH);
    \coordinate (lU) at ({-0.5*\patchL},{2.0*\patchH});
    \coordinate (lV) at ({0.5*\patchL},{2.0*\patchH});
    \coordinate (lJ) at ({-1.5*\patchL},0);
    \coordinate (lI) at (-\patchL,-\patchH);
    \coordinate (lL) at (\patchL,-\patchH);
    \coordinate (lK) at ({1.5*\patchL},0);
    \coordinate (lPatchCenter) at ($1/3*(lA)+1/3*(lB)+1/3*(lC)$);

    \DrawBrokenTriangle{neightri}{lA}{lD}{lU}{lADU}
    \DrawBrokenTriangle{neightri}{lA}{lU}{lV}{lAUV}
    \DrawBrokenTriangle{neightri}{lA}{lV}{lE}{lAVE}

    \DrawBrokenTriangle{neightri}{lB}{lD}{lJ}{lBDJ}
    \DrawBrokenTriangle{neightri}{lB}{lJ}{lI}{lBJI}
    \DrawBrokenTriangle{neightri}{lB}{lI}{lF}{lBIF}

    \DrawBrokenTriangle{neightri}{lC}{lF}{lL}{lCFL}
    \DrawBrokenTriangle{neightri}{lC}{lL}{lK}{lCLK}
    \DrawBrokenTriangle{neightri}{lC}{lK}{lE}{lCKE}

    \DrawBrokenTriangle{neightri}{lA}{lB}{lD}{lABD}
    \DrawBrokenTriangle{neightri}{lA}{lC}{lE}{lACE}
    \DrawBrokenTriangle{neightri}{lB}{lC}{lF}{lBCF}

    \DrawBrokenTriangle{centtri}{lA}{lB}{lC}{lABC}

    \node[brokenvertexdof] at (lABCa) {};
    \node[brokenvertexdof] at (lABCb) {};
    \node[brokenvertexdof] at (lABCc) {};

    \node[brokenvertexdof] at (lABDa) {};
    \node[brokenvertexdof] at (lABDb) {};
    \node[brokenvertexdof] at (lACEa) {};
    \node[brokenvertexdof] at (lACEb) {};
    \node[brokenvertexdof] at (lBCFa) {};
    \node[brokenvertexdof] at (lBCFb) {};

    \node[brokenvertexdof] at (lADUa) {};
    \node[brokenvertexdof] at (lAUVa) {};
    \node[brokenvertexdof] at (lAVEa) {};

    \node[brokenvertexdof] at (lBDJa) {};
    \node[brokenvertexdof] at (lBJIa) {};
    \node[brokenvertexdof] at (lBIFa) {};

    \node[brokenvertexdof] at (lCFLa) {};
    \node[brokenvertexdof] at (lCLKa) {};
    \node[brokenvertexdof] at (lCKEa) {};

    \node[brokenedgedof] at ($(lABCa)!0.5!(lABCb)$) {};
    \node[brokenedgedof] at ($(lABDa)!0.5!(lABDb)$) {};

    \node[brokenedgedof] at ($(lABCa)!0.5!(lABCc)$) {};
    \node[brokenedgedof] at ($(lACEa)!0.5!(lACEb)$) {};

    \node[brokenedgedof] at ($(lABCb)!0.5!(lABCc)$) {};
    \node[brokenedgedof] at ($(lBCFa)!0.5!(lBCFb)$) {};

    \node[brokencelldof] at ($1/3*(lABCa)+1/3*(lABCb)+1/3*(lABCc)$) {};
\end{scope}
\end{scope}

\draw[copyarrow, bend left=55]
    (2.55,0.72) to
    node[midway, above=4pt] {$P^\top$}
    (3.8,0.72);

\draw[copyarrow, bend left=55]
    (3.8,-0.72) to
    node[midway, below=4pt] {$P$}
    (2.55,-0.72);

\begin{scope}[xshift=5.75cm]
    \node[paneltitle] at (0,\panelTitleY) {(b) conforming DOFs};

    \pgfmathsetmacro{\confPatchL}{\confPatchScale*\patchL}
    \pgfmathsetmacro{\confPatchH}{0.8660254038*\confPatchL}
    
    \coordinate (mA) at (0,\confPatchH);
    \coordinate (mB) at ({-0.5*\confPatchL},0);
    \coordinate (mC) at ({0.5*\confPatchL},0);
    \coordinate (mD) at (-\confPatchL,\confPatchH);
    \coordinate (mE) at (\confPatchL,\confPatchH);
    \coordinate (mF) at (0,-\confPatchH);
    \coordinate (mU) at ({-0.5*\confPatchL},{2.0*\confPatchH});
    \coordinate (mV) at ({0.5*\confPatchL},{2.0*\confPatchH});
    \coordinate (mJ) at ({-1.5*\confPatchL},0);
    \coordinate (mI) at (-\confPatchL,-\confPatchH);
    \coordinate (mL) at (\confPatchL,-\confPatchH);
    \coordinate (mK) at ({1.5*\confPatchL},0);

    \draw[neightri] (mA)--(mD)--(mU)--cycle;
    \draw[neightri] (mA)--(mU)--(mV)--cycle;
    \draw[neightri] (mA)--(mV)--(mE)--cycle;

    \draw[neightri] (mB)--(mD)--(mJ)--cycle;
    \draw[neightri] (mB)--(mJ)--(mI)--cycle;
    \draw[neightri] (mB)--(mI)--(mF)--cycle;

    \draw[neightri] (mC)--(mF)--(mL)--cycle;
    \draw[neightri] (mC)--(mL)--(mK)--cycle;
    \draw[neightri] (mC)--(mK)--(mE)--cycle;

    \draw[neightri] (mA)--(mB)--(mD)--cycle;
    \draw[neightri] (mA)--(mC)--(mE)--cycle;
    \draw[neightri] (mB)--(mC)--(mF)--cycle;

    \draw[centtri] (mA)--(mB)--(mC)--cycle;

    \node[confvertexdof] at (mA) {};
    \node[confvertexdof] at (mB) {};
    \node[confvertexdof] at (mC) {};

    \node[confedgedof] at ($(mA)!0.5!(mB)$) {};
    \node[confedgedof] at ($(mA)!0.5!(mC)$) {};
    \node[confedgedof] at ($(mB)!0.5!(mC)$) {};

    \node[confcelldof] at ($1/3*(mA)+1/3*(mB)+1/3*(mC)$) {};
\end{scope}

\begin{scope}[xshift=8.85cm]
    \pgfmathsetmacro{\rhsCenter}{1.60}

    \pgfmathsetmacro{\yvertex}{1.18}
    \pgfmathsetmacro{\yedge}{-0.32}
    \pgfmathsetmacro{\ycell}{-1.40}

    \pgfmathsetmacro{\vsep}{0.34}
    \pgfmathsetmacro{\esep}{0.23}

    \node[paneltitle, anchor=center] at (\rhsCenter,\panelTitleY)
        {(c) action of \(D^{-1} P^\top\)};

    \node[brokenvertexdof] at (0.15,{\yvertex+2.5*\vsep}) {};
    \node[brokenvertexdof] at (0.15,{\yvertex+1.5*\vsep}) {};
    \node[brokenvertexdof] at (0.15,{\yvertex+0.5*\vsep}) {};
    \node[brokenvertexdof] at (0.15,{\yvertex-0.5*\vsep}) {};
    \node[brokenvertexdof] at (0.15,{\yvertex-1.5*\vsep}) {};
    \node[brokenvertexdof] at (0.15,{\yvertex-2.5*\vsep}) {};
    \draw[copyarrow] (0.42,\yvertex) -- (1.27,\yvertex);
    \node[confvertexdof] at (1.52,\yvertex) {};
    \node[anchor=west] at (2.35,\yvertex) {$d_i=6$};

    \node[brokenedgedof] at (0.15,{\yedge+\esep}) {};
    \node[brokenedgedof] at (0.15,{\yedge-\esep}) {};
    \draw[copyarrow] (0.42,\yedge) -- (1.27,\yedge);
    \node[confedgedof] at (1.52,\yedge) {};
    \node[anchor=west] at (2.35,\yedge) {$d_i=2$};

    \node[brokencelldof] at (0.15,\ycell) {};
    \draw[copyarrow] (0.42,\ycell) -- (1.27,\ycell);
    \node[confcelldof] at (1.52,\ycell) {};
    \node[anchor=west] at (2.35,\ycell) {$d_i=1$};

    \node[align=center, text width=4.2cm] at (1.55,-2.15)
        {\(d_i=\#\{\)broken copies of DOF \(i\}\)};
\end{scope}

\begin{scope}[yshift=-3.18cm]
    \node[anchor=west] at (-2.05,0.00) {broken};
    \node[anchor=west] at (-2.05,-0.46) {conforming};

    \node[brokenvertexdof] at (0.95,0.00) {};
    \node[anchor=west] at (1.15,0.00) {vertex DOF};

    \node[confvertexdof] at (0.95,-0.46) {};
    \node[anchor=west] at (1.15,-0.46) {vertex DOF};

    \node[brokenedgedof] at (4.00,0.00) {};
    \node[anchor=west] at (4.20,0.00) {edge DOF};

    \node[confedgedof] at (4.00,-0.46) {};
    \node[anchor=west] at (4.20,-0.46) {edge DOF};

    \node[brokencelldof] at (6.65,0.00) {};
    \node[anchor=west] at (6.85,0.00) {cell-interior DOF};

    \node[confcelldof] at (6.65,-0.46) {};
    \node[anchor=west] at (6.85,-0.46) {cell-interior DOF};
\end{scope}

\end{tikzpicture}

}

\caption{Local illustration of the conforming-to-broken injection \(P\) for a
domain-interior triangular element.  The left panel shows the relevant broken
local copies on the central element and all of its vertex-neighboring elements,
drawn separated before conformity is imposed.  The middle panel shows the same
patch after identifying shared vertex and edge DOFs.  Each column of \(P\)
copies one conforming coefficient into all of its broken local copies, and
therefore \(D=P^\top P\) is diagonal with entries equal to the number of such
copies.  For the local patch shown here, the vertex DOFs have multiplicity
\(6\), interior-edge DOFs have multiplicity \(2\), and element-interior DOFs
have multiplicity \(1\).
}
\label{fig:P_example}
\end{figure}

The BRAS inverse approximation we present in this paper consists of inverting local element matrices in the broken space, and mapping them back to the global conforming space with a careful weighting by multiplicity:
\begin{equation}
\label{eq:mhatinvbras_def}
\boxed{
    \Mbras^{-1}
    =
    (P^\top P)^{-1}P^{\top}\Mbr^{-1}P (P^\top P)^{-1}.
    }
\end{equation}
Here $D := P^\top P \in \mathbb{R}^{n \times n}$ is a diagonal matrix, whose $i$th entry counts the number of broken copies of the $i$th conforming DOF,
\[
D=P^{\top}P=\diag(d_1,\ldots,d_n),
    \qquad
    d_i=\#\{j:\pi(j)=i\}.
\]
\begin{subequations}
Note that the inverse \(\Mbr^{-1} = \operatorname{diag}_{K\in\Th} M_K^{-1}\) is well defined, since element-local mass matrices are SPD.
Thus BRAS is assembled analogously to the conforming mass matrix itself, replacing each local mass block by its inverse and applying multiplicity scaling. As such, in principle any finite-element library would be able to explicitly assemble BRAS using the same kernels as it uses to assemble conforming mass matrices. Note that for full column-rank $P$, it follows immediately from \eqref{eq:mhatinvbras_def} that $\Mbras^{-1}$ is SPD. 

The BRAS inverse can also be written as a symmetrically diagonally scaled elementwise sum,
\label{eq:mass_and_bras_element_sums}
\begin{align}
    \Mbras^{-1}
    &=
    D^{-1}
    \left(
    \sum_{K\in\Th} P_K^{\top}M_K^{-1}P_K
    \right)
    D^{-1}.
\end{align}
\end{subequations}
Here we see that \(\Mbras^{-1}\) has the same nonzero pattern as \(M\). Therefore one application of the assembled BRAS matrix has the same local dense matrix-vector work as one application of the assembled mass matrix.

\subsection{Theory}
\label{subsec:broken_exact_inverse}

We define two SPD (or symmetric negative definite) matrices $A,B\in\mathbb{R}^{n\times n}$ in the matrix pencil $(A,B)$ to be spectrally equivalent if there exist constants $0<\alpha \leq \beta$ such that for all $\bm{u}\in\mathbb{R}^n\backslash \{\bm{0} \}$
\begin{equation}\label{eq:spec-equiv}
    \alpha \leq \frac{\langle A\bm{u}, \bm{u}\rangle }{\langle B\bm{u}, \bm{u}\rangle } \leq \beta.
\end{equation}
Let $\beta$ and $\alpha$ be the maximum and minimum constants such that \eqref{eq:spec-equiv} holds, respectively, corresponding to the maximum and minimum generalized eigenvalues of the matrix pencil $(A,B)$. Then we define the \emph{spectral condition number} (distinct from the more common $\ell^2$-condition number) of the matrix pencil $(A,B)$ as
\begin{equation}\label{eq:spectral-kappa}
    \wh{\kappa}(A,B) \coloneqq \frac{\beta}{\alpha} = \frac{\lambda_{\max}(B^{-1}A)}{\lambda_{\min}(B^{-1}A)},
\end{equation}
where $\wh{\kappa}(A,B) = \wh{\kappa}(B,A)$.
In this section we prove a sharp bound on $\alpha$ for the BRAS-preconditioned mass operator; as discussed in \cref{sec:mass_numerics}, this is useful in practice because certain fixed-point iterations require reliable spectral intervals. Moreover, we characterize $\beta$ through a broken-space projection problem; numerical results in \cref{sec:mass_numerics} suggest that \(\Mbras^{-1}\) may give spectral condition numbers bounded independently of, or weakly dependent on, \(p\) in certain cases, although we do not prove such a bound in this work.

Consider the conforming mass solve \(M\bm u=\bm f\) for \(\bm u \in \mathbb{R}^{n}\).
By \eqref{eq:M_PMbrP}, \(M=P^{\top}\Mbr P\), so this problem can also be posed in terms of broken coefficients $\bm w \in \mathbb{R}^{n_{\rm br}}$ satisfying $P^\top \Mbr \bm w = \bm f$ subject to the constraint that  \(\bm{w} \in \range(P)\). Thus, let  \({\cal C} \in \mathbb{R}^{m \times n_{\rm br}}\) be a full-row-rank constraint
matrix on the broken coefficient space such that \(\kerop ({\cal C})=\range(P)\), wherein
\({\cal C}\bm w=\bm 0\) exactly enforces that the broken copies are consistent with a
single conforming coefficient vector. 
Here the number of constraints is \(m = n_{\rm br} - n\).
Define the broken right-hand side as \(\bm f_{\rm br} = P D^{-1}\bm f\),
where \(D=P^{\top}P\), and note that this lift has the correct assembled action because
\(P^{\top}P D^{-1}\bm f=\bm f\).  The conforming solve \(M\bm u=\bm f\) can therefore be written as the constrained broken-space saddle-point problem
\begin{equation}
\label{eq:broken_kkt}
    \begin{bmatrix}
        \Mbr & {\cal C}^{\top} \\
        {\cal C}   & 0
    \end{bmatrix}
    \begin{bmatrix}
        \bm w \\
        \bm\mu
    \end{bmatrix}
    =
    \begin{bmatrix}
        \bm f_{\rm br} \\
        \bm 0
    \end{bmatrix}.
\end{equation}
Indeed, the constraint gives \(\bm w=P\bm u\) for some conforming vector
\(\bm u\).  Multiplying the first block row of \eqref{eq:broken_kkt} by
\(P^{\top}\) eliminates the multiplier term, since \({\cal C}P=0\), and gives
\(P^{\top}\Mbr P\bm u=\bm f\), which is exactly the conforming system. The following Lemma uses this formulation to directly relate $\Mbras^{-1}$ to $M^{-1}$.

\begin{lemma}
\label{lem:bras_exact_inverse_correction}
From \eqref{eq:broken_kkt} define the Schur complement $S_{\cal C} \coloneqq -{\cal C}\Mbr^{-1}{\cal C}^{\top} \prec 0$. The BRAS inverse satisfies
\begin{equation}
\label{eq:bras_exact_inverse_correction}
    \Mbras^{-1}
    =
    M^{-1}
    -
    Z_{\cal C} S_{\cal C}^{-1} Z_{\cal C}^{\top},
    \qquad\textnormal{where }
    Z_{\cal C} = D^{-1}P^{\top}\Mbr^{-1}{\cal C}^{\top},
\end{equation}
and \(-Z_{\cal C} S_{\cal C}^{-1} Z_{\cal C}^{\top} \succeq 0\).
\end{lemma}

\begin{proof}
Solving \eqref{eq:broken_kkt} for \(\bm\mu\) gives
\(
\bm\mu = -S_{\cal C}^{-1}{\cal C} \Mbr^{-1} \bm f_{\rm br}.
\)
Plugging this into 
\eqref{eq:broken_kkt} and solving for \(\bm w\) gives
\[
    \bm w
    =
    \left(
        \Mbr^{-1}
        +
        \Mbr^{-1}{\cal C}^{\top}S_{\cal C}^{-1}{\cal C}\Mbr^{-1}
    \right)
    \bm f_{\rm br}.
\]
Since \( \bm w \in \range(P)\), there is a conforming coefficient vector \(\bm u\) such that
\(\bm w=P\bm u\). Hence, applying \(D^{-1}P^{\top}\) gives
\[
    \bm u
    =
    \bigl[
    D^{-1}P^{\top}
    \left(
        \Mbr^{-1}
        +
        \Mbr^{-1}{\cal C}^{\top}S_{\cal C}^{-1}{\cal C}\Mbr^{-1}
    \right)
    P D^{-1}
    \bigr]\bm f .
\]
The previous paragraph shows that this \(\bm u\) is the solution of
\(M\bm u=\bm f\), hence
\[
    M^{-1}
    =
    D^{-1}P^{\top}\Mbr^{-1}P D^{-1}
    +
    Z_{\cal C} S_{\cal C}^{-1} Z_{\cal C}^{\top}.
\]
Using the definition
\(
    \Mbras^{-1}=D^{-1}P^{\top}\Mbr^{-1}P D^{-1}
\)
and rearranging gives \eqref{eq:bras_exact_inverse_correction}.

\end{proof}

\begin{corollary}
\label{cor:bras_lambda_min}
The BRAS-preconditioned mass operator satisfies
\begin{equation}
\label{eq:bras_lambda_min}
    \lambda_{\min}(\Mbras^{-1} M) \geq 1.
\end{equation}
\end{corollary}

\begin{proof}
By \cref{lem:bras_exact_inverse_correction}, the difference
\(\Mbras^{-1}-M^{-1} = -Z_{\cal C} S_{\cal C}^{-1}Z_{\cal C}^{\top} \succeq 0\) is symmetric positive semidefinite (SPSD).  Hence
\(\Mbras^{-1} \succeq M^{-1}\).  Congruence by \(M^{1/2}\) gives
\(M^{1/2}\Mbras^{-1}M^{1/2}\succeq I\), and this matrix is similar to
\(\Mbras^{-1}M\).
\end{proof}

The lower bound in \eqref{eq:bras_lambda_min} can be sharp; numerically, we find that equality holds in many cases, and even when it does not we often have $\lambda_{\min} \approx 1$.
Let
\(\Delta_{\cal C}:=-Z_{\cal C}S_{\cal C}^{-1}Z_{\cal C}^{\top}\), so that
\(\Mbras^{-1}=M^{-1}+\Delta_{\cal C}\) and \(\Delta_{\cal C}\succeq0\).  Then
\(\lambda_{\min}(\Mbras^{-1}M)=1\) precisely when the correction \(\Delta_{\cal C}\)
has a non-trivial nullspace, and there are at least two mechanisms by which this can occur, with the first being purely
dimensional.
Since \(-S_{\cal C}^{-1}\) is SPD,
\(\kerop (\Delta_{\cal C})=\kerop (Z_{\cal C}^{\top})\).  
Since \(Z_{\cal C}^{\top}\in\mathbb R^{m\times n}\), rank-nullity implies that
\(\kerop Z_{\cal C}^{\top}\) is non-trivial whenever \(m<n\).  Recalling that
\(m=n_{\rm br}-n\), this condition is equivalent to \(n_{\rm br}<2n\).
To make this more concrete, if the
\(i\)th conforming DOF has \(d_i\geq1\) broken copies, then
\[
    n_{\rm br}=\sum_i d_i,
    \quad
    n=\sum_i 1
\qquad
\Longrightarrow
\qquad
    m-n
    =
    n_{\rm br}-2n
    =
    \sum_i(d_i-2).
\]
Hence, degrees of freedom (DOFs) with multiplicity $d_i = 1$ contribute negatively towards \(m - n\), DOFs with multiplicity $d_i = 2$ contribute neutrally towards \(m - n\), and DOFs with $d_i > 2$ contribute positively towards \(m - n\).
Natural candidates for \(m<n\) are therefore spaces whose DOFs are concentrated on
entities with low sharing multiplicity: cell-interior and boundary-facet DOFs have \(d_i=1\), while interior-facet DOFs (edges in 2D, faces in 3D) typically have \(d_i=2\).
This makes \(\Hdiv\) a natural candidate, where DOFs are typically on facets and cell interiors.
By contrast, low-order \(H^1\) spaces in two and three dimensions often place DOFs on shared vertices, for which \(d_i>2\) typically, so the dimensional argument may fail.

A second mechanism for \(\lambda_{\min}(\Mbras^{-1} M)=1\) occurs when
\(\Mbr P\bm1 \in\range(P)\).  Let \(Q:=P D^{-1}P^{\top}\) be the
\(\ell^2\)-orthogonal projector onto \(\range(P)\) (see also
\cref{thm:bras_lambda_max_projector} below).  If
\(\Mbr P\bm1\in\range(P)\), then \(Q\Mbr P\bm1=\Mbr P\bm1\), and hence
\[
    \Mbras^{-1}M\bm1
    =
    D^{-1}P^{\top}\Mbr^{-1}Q\Mbr P\bm1
    =
    D^{-1}P^{\top}\Mbr^{-1}\Mbr P\bm1
    =
    D^{-1}P^{\top}P\bm1
    =
    \bm1 .
\]
A sufficient condition for \(\Mbr P\bm1\in\range(P)\) is that
\(M_K\bm1_K=c\bm1_K\) with \(c\) independent of \(K\), since then
\(\Mbr P\bm1=cP\bm1\).  This arises, for example, for scalar \(P_1\)
nodal elements on uniform affine simplicial meshes, for which
\(c=|K|/(d+1)\) is element-independent.

\begin{theorem}
\label{thm:bras_lambda_max_projector}
Let
\begin{equation}
\label{eq:broken_l2_projector}
    Q = P ( P^\top P )^{-1} P^{\top}
\end{equation}
be the $\ell^2$-orthogonal projector onto \(\range(P)\).  Then
\begin{equation}
\label{eq:bras_lambda_max_projector}
    \wh{\kappa}(M, \Mbras ) 
    \leq
    \lambda_{\max}(\Mbras^{-1} M)
    =
    \|Q\|_{\Mbr^{-1}}^2
    =
    \lambda_{\max}\left( \Mbr ( Q\Mbr^{-1}Q) \right).
\end{equation}
\end{theorem}
\begin{proof}
Since \(Q^{\top}=Q\) and \(\Mbr^{-1}\) is SPD, by definition, we have
\[
\|Q\|_{\Mbr^{-1}}^2
    :=
    \sup_{\bm{x} \neq \bm 0}
    \frac{\bm{x}^{\top}Q\Mbr^{-1}Q\bm{x}}
         {\bm{x}^{\top}\Mbr^{-1}\bm{x}}
    =
    \lambda_{\max}\left( \Mbr ( Q\Mbr^{-1}Q) \right).
\]
Using \(Q=PD^{-1}P^{\top}\), we factor the previous matrix as
\[
    \Mbr Q\Mbr^{-1}Q
    =
    (\Mbr P)
    (D^{-1}P^{\top}\Mbr^{-1}P D^{-1}P^{\top}) 
    =  
    (\Mbr P)
    (\Mbras^{-1} P^{\top}) 
    \equiv A B.
\]
The nonzero eigenvalues of \(AB\) are the nonzero eigenvalues of \(BA\).  In this
case,
\[
    BA
    =
    (\Mbras^{-1} P^{\top}) (\Mbr P)
    =
    \Mbras^{-1}M.
\]
The eigenvalues of $BA$ are thus exactly the generalized eigenvalues of \(M\bm{u}=\lambda\Mbras\bm{u}\), which proves \eqref{eq:bras_lambda_max_projector}. 
The spectral condition number result follows immediately from \Cref{cor:bras_lambda_min} and \eqref{eq:spectral-kappa}.
\end{proof}

The viewpoint described above surrounding \eqref{eq:broken_kkt} is also related to domain decomposition methods such as
finite-element tearing and interconnecting (FETI) and balancing domain
decomposition by constraints (BDDC), which introduce duplicated interface
variables and constraints between nonoverlapping subdomains; see
\cite{mandel2005substructure}.  
These methods are also closely connected to the algebraic hybridization approach of Dobrev et al.~\cite{dobrev2019hybridization}.

\subsection{Piola mappings and metric-dependent BRAS bounds}
\label{sec:piola}

Let \(F_K:\widehat K\to K\) denote a sufficiently smooth, orientation-preserving map from a reference cell \(\widehat K\) to a physical mesh cell \(K\), with Jacobian matrix \(J_K\) and corresponding determinant \(\det (J_K)\).  
We write \(\bm x = F_K(\widehat{\bm x})\), and suppress the dependence of \(J_K\) on \(\widehat{\bm x}\).
Element mass matrices are computed by mapping physical-cell \(L^2(K)\) inner products back to reference-cell integrals.
Reference-cell integrands naturally inherit a factor of \(\det (J_K)\) due to this change of variables, \(\ud \bm x = \det(J_K)  \ud \bm{\wh{x}}\).
Furthermore, the reference-cell integrand depends on the finite-element space, because the reference-to-physical map must preserve the trace quantity used to impose conformity; see \cite[Section 2.1.3]{boffi2013mixed}. 
For \(H(\curl)\) elements, continuity of tangential traces is enforced, leading to the covariant Piola map,  
\(
\bm{\Phi}( \bm{x} ) = J_K^{-\top} \wh{\bm \Phi}( \wh{\bm x} )
\),
with $\bm{\Phi}(\bm{x})$ a vector field on the physical cell and $\bm{\wh \Phi}(\bm{\wh x})$ its analog on the reference cell.
For \(H(\div)\) elements, continuity of normal fluxes is enforced, leading to the contravariant Piola map, 
\(
\bm{\Phi}( \bm{x} ) = \tfrac{1}{\det (J_K)} J_K \wh{\bm \Phi}( \wh{\bm x} ).
\)

Applying the above change of variables, we find that corresponding physical mass inner products on a single element pull back as follows:
\begin{subequations}
\label{eq:piola_mass_pullbacks}
\begin{align}
    (u,v)_{L^2(K)}
    &=
    \int_{\widehat K}
    \widehat u\, \det (J_K)\, \widehat v\,\ud\widehat x,
    && u,v\in H^1(K), \label{eq:h1_mass_pullback} \\
    (\bm u,\bm v)_{L^2(K)}
    &=
    \int_{\widehat K}
    \widehat{\bm u}^{\top}
    \left(\det (J_K)J_K^{-1}J_K^{-\top}\right)
    \widehat{\bm v}\,\ud\widehat{\bm{x}},
    && \bm u,\bm v\in H(\curl;K), \label{eq:hcurl_mass_pullback} \\
    (\bm u,\bm v)_{L^2(K)}
    &=
    \int_{\widehat K}
    \widehat{\bm u}^{\top}
    \left(\frac{J_K^\top J_K}{\det (J_K)}\right)
    \widehat{\bm v}\,\ud\widehat{\bm{x}},
    && \bm u,\bm v\in H(\div;K). \label{eq:hdiv_mass_pullback}
\end{align}
\end{subequations}
Thus scalar \(H^1\) mass matrices depend on the element map only through the scalar Jacobian factor \(\det (J_K)\), while \(H(\curl)\) and \(H(\div)\) mass matrices depend on tensor-valued Piola metric weights.  

The following results compare these metric effects algebraically.  We use \(G_K\) to denote the scalar or tensor coefficient that appears in the pulled-back reference-cell mass integrand in \cref{eq:piola_mass_pullbacks}. Let \(M^I\) denote the mass matrix assembled on the reference cells with the identity coefficient, and let \(M^G\) denote the mass matrix assembled on the same reference finite element space with coefficients \(G_K\). Equivalently, \(M^G\) is the physical mass matrix on the mapped mesh. The matrices \(\Mbr^I\) and \(\Mbr^G\) denote the corresponding unassembled mass matrices on the associated broken space.

\begin{lemma}
\label{lem:bras_metric_endpoint}
Let \(M^I=P^\top \Mbr^I P\) and \(M^G=P^\top \Mbr^G P\), and let
\((\Mbras^I)^{-1}\) and \((\Mbras^G)^{-1}\) denote the corresponding BRAS
inverse approximations.  
Assume that $(\Mbr^I, \Mbr^G)$ are spectrally equivalent as in \eqref{eq:spec-equiv} with constants $0 < \alpha \leq \beta$. Then
\[
    \lambda_{\max}\!\left((\Mbras^G)^{-1}M^G\right)
    \leq
    \frac{\beta}{\alpha}
    \lambda_{\max}\!\left((\Mbras^I)^{-1}M^I\right).
\]
\end{lemma}

\begin{proof}
The spectral-equivalence assumption implies
\[
    \frac{1}{\beta}(\Mbr^I)^{-1}
    \preceq
    (\Mbr^G)^{-1}
    \preceq
    \frac{1}{\alpha}(\Mbr^I)^{-1}.
\]
Since the conforming-to-broken map is the same in both cases, the projector
\(Q=P(P^\top P)^{-1}P^\top\) is also the same.  Using the characterization
from \cref{thm:bras_lambda_max_projector},
\[
\begin{aligned}
    \lambda_{\max}\!\left((\Mbras^G)^{-1}M^G\right)
    &=
    \sup_{\bm x\neq\bm 0}
    \frac{\bm x^\top Q(\Mbr^G)^{-1}Q\bm x}
         {\bm x^\top(\Mbr^G)^{-1}\bm x} \\
    &\leq
    \frac{\beta}{\alpha}
    \sup_{\bm x\neq\bm 0}
    \frac{\bm x^\top Q(\Mbr^I)^{-1}Q\bm x}
         {\bm x^\top(\Mbr^I)^{-1}\bm x} \\
    &=
    \frac{\beta}{\alpha}
    \lambda_{\max}\!\left((\Mbras^I)^{-1}M^I\right).
\end{aligned}
\]
\end{proof}

\begin{corollary}
\label{cor:bras_piola_metric_contrast}
Assume that there exist constants \(0<\alpha\leq\beta\) such that
\[
    \alpha I \preceq G_K(\wh{\bm{x}}) \preceq \beta I
    \qquad\text{for all }K\in\Th
\]
and almost every \(\wh{\bm x}\in\widehat K\).  Then the corresponding BRAS
operators satisfy
\[
    \lambda_{\max}\!\left((\Mbras^G)^{-1}M^G\right)
    \leq
    \frac{\beta}{\alpha}
    \lambda_{\max}\!\left((\Mbras^I)^{-1}M^I\right).
\]
If, in addition, the lower-bound estimate in \cref{cor:bras_lambda_min} is
sharp for both the \(I\)- and \(G\)-metric operators, so that
\(
    \lambda_{\min}\!\left((\Mbras^G)^{-1}M^G\right)
    =
    \lambda_{\min}\!\left((\Mbras^I)^{-1}M^I\right)
    =
    1,
\)
then
\[
    \wh{\kappa}\!\left(M^G,\Mbras^G\right)
    \leq
    \frac{\beta}{\alpha}
    \wh{\kappa}\!\left(M^I,\Mbras^I\right).
\]
\end{corollary}

\begin{proof}
The pointwise metric bounds imply the corresponding elementwise mass bounds
\[
    \alpha M_K^I \preceq M_K^G \preceq \beta M_K^I
    \qquad\text{for all }K\in\Th .
\]
Since the broken mass matrices are assembled blockwise over elements, this gives
\[
    \alpha \Mbr^I \preceq \Mbr^G \preceq \beta \Mbr^I .
\]
The BRAS maximum eigenvalue bounds then follow directly from
\cref{lem:bras_metric_endpoint}, and the spectral condition number result follows by definition from \eqref{eq:spectral-kappa}.
\end{proof}

In particular, for the pullbacks in \cref{eq:piola_mass_pullbacks}, the
metric contrast \(\beta/\alpha\) may be bounded by
\[
    \frac{\beta}{\alpha}
    \leq
    \begin{cases}
        \displaystyle
        \frac{\sup_{K\in\Th} \det (J_K)}{\inf_{K\in\Th} \det (J_K)},
        & H^1, \\[1.25ex]
        \displaystyle
        \frac{\sup_{K\in\Th}\lambda_{\max}(\det (J_K)J_K^{-1}J_K^{-\top})}
             {\inf_{K\in\Th}\lambda_{\min}(\det (J_K)J_K^{-1}J_K^{-\top})},
        & H(\curl), \\[2ex]
        \displaystyle
        \frac{\sup_{K\in\Th}\lambda_{\max}(J_K^\top J_K/\det (J_K))}
             {\inf_{K\in\Th}\lambda_{\min}(J_K^\top J_K/\det (J_K))},
        & H(\div).
    \end{cases}
\]

\section{Numerical results: Mass matrices}
\label{sec:mass_numerics}

This section studies solve performance for the mass system
\(M\bm{u}=\bm{f}\) using preconditioned CG and damped fixed-point iterations. Although consistently slower than CG as a standalone solver, fixed-point iterations can still be useful to avoid the parallel communication costs associated with CG dot products or inside of larger coupled block preconditioners. In addition, our interest is not only in solving an isolated mass system, but also in using BRAS as an approximate inverse inside larger preconditioners; see \cref{sec:block_preconditioning,sec:block_preconditioning_numerics}.  
In that setting, the standalone condition number, and the fixed-point behavior it controls, are often a cleaner measure of preconditioner quality than CG iterations. 

Let \(\Mwtinv\) denote the inverse approximation used as a preconditioner. Given a damping parameter \(\omega>0\), the fixed-point iteration is
\[
    \bm{u}_{k+1}
    =
    \bm{u}_k
    +
    \omega \Mwtinv
    \left(\bm{f} - M\bm{u}_k\right),
    \qquad
    \Mwtinv
    \in
    \bigl\{ \Mdiag^{-1}, \, \Mbras^{-1}
    \bigr\}.
\]
The fixed-point iteration is convergent for any
\(\omega\in(0,2/\lambda_{\max}(\Mwtinv M))\).  The choice minimizing the
spectral radius of the error propagation matrix \(I-\omega\Mwtinv M\) is
\[
    \omega_{\rm opt}(\Mwtinv)
    =
    \frac{2}
    {\lambda_{\min}(\Mwtinv M)+\lambda_{\max}(\Mwtinv M)} .
\]
Thus any convergent choice requires an estimate of
\(\lambda_{\max}(\Mwtinv M)\), while the optimal choice also requires an estimate of \(\lambda_{\min}(\Mwtinv M)\).  In general, maximum
eigenvalues are much cheaper to estimate than minimum eigenvalues: power and
Arnoldi-type iterations are usually sufficient for the former, while the latter often require inverse iterations and therefore solves with \(M\) shifted by a diagonal matrix.  
This motivates the following pragmatic weight choices used in the experiments here:
\[
    \omega_{\rm diag}
    =
    \frac{1}{\lambda_{\max}(\Mdiag^{-1} M)},
    \qquad
    \omega_{\rm bras}
    =
    \frac{2}{1+\lambda_{\max}(\Mbras^{-1} M)} .
\]
For BRAS this choice often yields the optimal weight or not far from it, using the lower spectral bound
\(\lambda_{\min}(\Mbras^{-1} M) \geq 1\) from
\cref{cor:bras_lambda_min} that we note is generally fairly tight. For the diagonal method we use only the upper
spectral estimate, since we are unaware of any analogous lower spectral estimate. 
The largest eigenvalue in both weights is estimated numerically from the corresponding preconditioned operator using the power iteration.

All finite elements are implemented using the Firedrake library
\cite{FiredrakeUserManual}, and we use Firedrake element-family names in the tables and text.
For \(H^1\), ``Lagrange'' and ``Bernstein'' denote
\(H^1\)-conforming piecewise-\(P_p\) spaces represented using nodal and Bernstein bases, respectively; the nodal basis uses Firedrake's spectral variant, based on Gauss--Lobatto--Legendre (GLL) points.
For \(H(\operatorname{curl})\),
``\(\mathrm{N1curl}\)'' and ``\(\mathrm{N2curl}\)'' denote the first- and second-kind N\'ed\'elec  elements, respectively, and for \(H(\operatorname{div})\)
``\(\mathrm{RT}\)'' and ``\(\mathrm{BDM}\)'' denote the Raviart--Thomas and
Brezzi--Douglas--Marini elements, respectively. 
All meshes use simplices (triangles in 2D and tetrahedra in 3D). For solve tests, the number of mesh elements is chosen so that the conforming problem has approximately \(10^6\) DOFs, while spectral condition number calculations use a smaller problem with approximately \(5\times 10^3\) DOFs to enable direct calculation. Solves are done on a random right-hand side \(\bm{f}\).  
Both CG and fixed-point iterations are initialized with \(\bm{u}_0=\bm{0}\) and are terminated when the residual \(\bm{r}_k=\bm{f}-M\bm{u}_k\) satisfies
\(\|\bm{r}_k\|_2/\|\bm{r}_0\|_2 \le 10^{-10}\), or after \(5000\)
iterations.  Solve times do not include construction of the preconditioner or
estimation of \(\lambda_{\max}\).

\providecommand{\fail}[1]{#1^{*}}

\begin{table}[t!]
\centering
\begingroup
\scriptsize
\setlength{\tabcolsep}{1.8pt}
\renewcommand{\arraystretch}{0.95}
\begin{tabular*}{\textwidth}{@{\extracolsep{\fill}} c c cc cc c cc cc @{}}
\toprule
\multicolumn{11}{c}{Space \(=H^1\)} \\
\midrule
&
\multicolumn{5}{c}{\(\mathrm{Lagrange}\)}
&
\multicolumn{5}{c}{\(\mathrm{Bernstein}\)} \\
\cmidrule(lr){2-6}\cmidrule(lr){7-11}
\(p\)
& \(\hat{\kappa}\)
& \multicolumn{2}{c}{Fixed point}
& \multicolumn{2}{c}{CG}
& \(\hat{\kappa}\)
& \multicolumn{2}{c}{Fixed point}
& \multicolumn{2}{c}{CG} \\
\cmidrule(lr){3-4}\cmidrule(lr){5-6}
\cmidrule(lr){8-9}\cmidrule(lr){10-11}
&
&
Its. & \(t\)
& Its. & \(t\)
&
&
Its. & \(t\)
& Its. & \(t\) \\
\midrule
\multicolumn{11}{c}{Two dimensions} \\
\midrule
1 & \(4.0/1.6\) & \(73/15\) & \(0.81/0.23\) & \(22/11\) & \(0.32/0.21\)
  & \(4.0/1.6\) & \(73/15\) & \(0.80/0.23\) & \(22/11\) & \(0.29/0.21\) \\
2 & \(5.2/1.8\) & \(99/18\) & \(1.3/0.39\) & \(25/13\) & \(0.41/0.30\)
  & \(16/3.8\) & \(310/43\) & \(4.1/0.87\) & \(45/21\) & \(0.74/0.49\) \\
3 & \(5.5/1.7\) & \(100/18\) & \(1.7/0.50\) & \(26/12\) & \(0.53/0.39\)
  & \(58/5.7\) & \(1175/65\) & \(20/1.9\) & \(90/27\) & \(1.8/0.82\) \\
4 & \(5.5/1.7\) & \(103/17\) & \(2.2/0.66\) & \(26/12\) & \(0.62/0.47\)
  & \(230/5.0\) & \(4575/59\) & \(100/2.3\) & \(180/23\) & \(4.5/0.96\) \\
\midrule
\multicolumn{11}{c}{Three dimensions} \\
\midrule
1 & \(5.0/1.8\) & \(91/18\) & \(1.5/0.56\) & \(24/12\) & \(0.46/0.33\)
  & \(5.0/1.8\) & \(91/18\) & \(1.5/0.53\) & \(24/12\) & \(0.45/0.35\) \\
2 & \(17/3.0\) & \(330/31\) & \(8.5/1.4\) & \(48/18\) & \(1.3/0.81\)
  & \(24/4.9\) & \(447/56\) & \(11/2.5\) & \(56/25\) & \(1.5/1.2\) \\
3 & \(13/2.9\) & \(243/31\) & \(11/2.4\) & \(41/18\) & \(1.7/1.5\)
  & \(98/7.8\) & \(1909/91\) & \(75/6.8\) & \(117/32\) & \(5.0/2.4\) \\
4 & \(22/3.5\) & \(397/38\) & \(24/4.4\) & \(52/20\) & \(3.4/2.4\)
  & \(420/11\) & \(\fail{5000}/134\) & \(\fail{300}/15\) & \(244/39\) & \(15/4.6\) \\
\bottomrule
\end{tabular*}
\endgroup
\caption{Mass-solve performance for \(H^1\) spaces. Entries are diagonal/BRAS. The columns report the estimated spectral condition number \(\hat{\kappa}\), iteration counts, and solve times \(t\) for fixed-point and CG solves. Starred entries reached the \(5000\)-iteration limit.}
\label{tab:h1_mass_solve}
\end{table}

\providecommand{\fail}[1]{#1^{*}}

\begin{table}[b!]
\centering
\begingroup
\scriptsize
\setlength{\tabcolsep}{1.8pt}
\renewcommand{\arraystretch}{0.95}
\begin{tabular*}{\textwidth}{@{\extracolsep{\fill}} c c cc cc c cc cc @{}}
\toprule
\multicolumn{11}{c}{Space \(=H(\operatorname{curl})\)} \\
\midrule
&
\multicolumn{5}{c}{\(\mathrm{N1curl}\)}
&
\multicolumn{5}{c}{\(\mathrm{N2curl}\)} \\
\cmidrule(lr){2-6}\cmidrule(lr){7-11}
\(p\)
& \(\hat{\kappa}\)
& \multicolumn{2}{c}{Fixed point}
& \multicolumn{2}{c}{CG}
& \(\hat{\kappa}\)
& \multicolumn{2}{c}{Fixed point}
& \multicolumn{2}{c}{CG} \\
\cmidrule(lr){3-4}\cmidrule(lr){5-6}
\cmidrule(lr){8-9}\cmidrule(lr){10-11}
&
&
Its. & \(t\)
& Its. & \(t\)
&
&
Its. & \(t\)
& Its. & \(t\) \\
\midrule
\multicolumn{11}{c}{Two dimensions} \\
\midrule
1 & \(3.0/1.3\) & \(54/12\) & \(0.48/0.15\) & \(18/9\) & \(0.22/0.13\)
  & \(13/3.4\) & \(240/38\) & \(2.9/0.75\) & \(39/19\) & \(0.61/0.41\) \\
2 & \(9.6/1.7\) & \(180/18\) & \(2.5/0.41\) & \(32/12\) & \(0.54/0.29\)
  & \(11/2.7\) & \(222/30\) & \(4.0/0.85\) & \(39/17\) & \(0.84/0.55\) \\
3 & \(14/2.0\) & \(269/22\) & \(5.3/0.71\) & \(41/14\) & \(0.91/0.55\)
  & \(20/3.0\) & \(391/33\) & \(9.5/1.4\) & \(51/18\) & \(1.3/0.82\) \\
4 & \(18/2.5\) & \(361/28\) & \(9.9/1.4\) & \(48/16\) & \(1.4/0.87\)
  & \(31/3.4\) & \(629/40\) & \(20/2.5\) & \(66/20\) & \(2.4/1.3\) \\
\midrule
\multicolumn{11}{c}{Three dimensions} \\
\midrule
1 & \(6.0/1.8\) & \(107/17\) & \(1.9/0.50\) & \(27/12\) & \(0.55/0.39\)
  & \(36/7.4\) & \(638/82\) & \(18/4.2\) & \(68/31\) & \(2.1/1.7\) \\
2 & \(83/6.4\) & \(1388/72\) & \(47/4.5\) & \(102/29\) & \(3.7/1.9\)
  & \(170/10\) & \(3148/119\) & \(170/12\) & \(151/38\) & \(8.2/3.9\) \\
3 & \(140/8.4\) & \(2590/102\) & \(180/13\) & \(137/35\) & \(9.6/4.8\)
  & \(210/13\) & \(3793/157\) & \(330/27\) & \(170/45\) & \(15/8.0\) \\
4 & \(280/13\) & \(\fail{5000}/168\) & \(\fail{540}/35\) & \(197/46\) & \(22/10\)
  & \(450/20\) & \(\fail{5000}/245\) & \(\fail{690}/65\) & \(252/56\) & \(36/15\) \\
\bottomrule
\end{tabular*}
\endgroup
\caption{Mass-solve performance for \(H(\operatorname{curl})\) spaces. Entries are diagonal/BRAS. The columns report the estimated spectral condition number \(\hat{\kappa}\), iteration counts, and solve times \(t\) for fixed-point and CG solves. Starred entries reached the \(5000\)-iteration limit.}
\label{tab:hcurl_mass_solve}
\end{table}

\providecommand{\fail}[1]{#1^{*}}

\begin{table}[t!]
\centering
\begingroup
\scriptsize
\setlength{\tabcolsep}{1.8pt}
\renewcommand{\arraystretch}{0.95}
\begin{tabular*}{\textwidth}{@{\extracolsep{\fill}} c c cc cc c cc cc @{}}
\toprule
\multicolumn{11}{c}{Space \(=H(\operatorname{div})\)} \\
\midrule
&
\multicolumn{5}{c}{\(\mathrm{RT}\)}
&
\multicolumn{5}{c}{\(\mathrm{BDM}\)} \\
\cmidrule(lr){2-6}\cmidrule(lr){7-11}
\(p\)
& \(\hat{\kappa}\)
& \multicolumn{2}{c}{Fixed point}
& \multicolumn{2}{c}{CG}
& \(\hat{\kappa}\)
& \multicolumn{2}{c}{Fixed point}
& \multicolumn{2}{c}{CG} \\
\cmidrule(lr){3-4}\cmidrule(lr){5-6}
\cmidrule(lr){8-9}\cmidrule(lr){10-11}
&
&
Its. & \(t\)
& Its. & \(t\)
&
&
Its. & \(t\)
& Its. & \(t\) \\
\midrule
\multicolumn{11}{c}{Two dimensions} \\
\midrule
1 & \(3.0/1.3\) & \(54/12\) & \(0.49/0.14\) & \(18/9\) & \(0.22/0.13\)
  & \(13/3.4\) & \(240/38\) & \(3.0/0.74\) & \(39/19\) & \(0.58/0.46\) \\
2 & \(9.6/1.7\) & \(180/18\) & \(2.4/0.37\) & \(32/12\) & \(0.54/0.29\)
  & \(11/2.7\) & \(222/30\) & \(3.8/0.85\) & \(39/17\) & \(0.77/0.51\) \\
3 & \(14/2.0\) & \(269/22\) & \(5.3/0.75\) & \(41/14\) & \(0.94/0.54\)
  & \(20/3.0\) & \(391/34\) & \(9.3/1.4\) & \(51/18\) & \(1.4/0.80\) \\
4 & \(18/2.5\) & \(361/28\) & \(9.7/1.4\) & \(48/16\) & \(1.5/0.81\)
  & \(31/3.4\) & \(629/40\) & \(20/2.4\) & \(66/20\) & \(2.4/1.2\) \\
\midrule
\multicolumn{11}{c}{Three dimensions} \\
\midrule
1 & \(4.0/1.6\) & \(77/15\) & \(0.80/0.22\) & \(22/11\) & \(0.30/0.20\)
  & \(39/7.1\) & \(720/79\) & \(14/2.6\) & \(69/30\) & \(1.6/1.1\) \\
2 & \(27/2.3\) & \(497/26\) & \(10/0.96\) & \(59/16\) & \(1.4/0.60\)
  & \(32/5.4\) & \(657/63\) & \(25/4.3\) & \(67/26\) & \(2.7/1.9\) \\
3 & \(38/3.0\) & \(701/35\) & \(28/2.7\) & \(71/19\) & \(3.1/1.5\)
  & \(210/5.7\) & \(4026/69\) & \(260/8.2\) & \(169/28\) & \(12/3.6\) \\
4 & \(49/3.7\) & \(916/44\) & \(66/5.9\) & \(82/22\) & \(6.5/3.2\)
  & \(310/6.5\) & \(\fail{5000}/79\) & \(\fail{510}/16\) & \(215/30\) & \(22/6.1\) \\
\bottomrule
\end{tabular*}
\endgroup
\caption{Mass-solve performance for \(H(\operatorname{div})\) spaces. Entries are diagonal/BRAS. The columns report the estimated spectral condition number \(\hat{\kappa}\), iteration counts, and solve times \(t\) for fixed-point and CG solves. Starred entries reached the \(5000\)-iteration limit.}
\label{tab:hdiv_mass_solve}
\end{table}

\Cref{tab:hdiv_mass_solve,tab:hcurl_mass_solve,tab:h1_mass_solve} show spectral condition numbers, solve times, and iteration counts of $\Mdiag$ and $\Mbras$ for two sets of basis functions associated with $H^1, H(\operatorname{curl})$ and $H(\operatorname{div})$, respectively, and finite-element orders $p\in[1,4]$. 
We report both solve time and iteration counts due to the different cost of applying $\Mdiag^{-1}$ vs. $\Mbras^{-1}$, with the latter costing the same as application of $M$. To compare, we consider speedup $S$ as the ratio of the $\Mdiag$-based solve execution time to the $\Mbras$-based solve execution time, and performance improvement $I$ as the percentage by which the solve is faster, where $I = 100(S-1)$. 
Throughout this paper, all reported solve times are mean values over three solves.
The tables show that BRAS reduces the spectral condition number across the tested spaces, with the largest relative improvements occurring in high-order spaces. Consistent with this, across all scenarios tested, preconditioning with $\Mbras^{-1}$ yields meaningful speedups over $\Mdiag^{-1}$. For CG solves, almost all cases see at least a 40\% speedup, with certain instances such as 4th-order Bernstein bases seeing a $4\times$ speedup  (or $\approx 300\%$ improvement in performance). Fixed-point improvements are even more dramatic, with $\Mbras^{-1}$ yielding at least a $3\times$ speedup (or $\approx 200\%$ improvement in performance) in almost all cases, and reaching as high as 20--40$\times$ faster in certain cases of 4th-order Bernstein and BDM elements. In several cases the BRAS fixed-point solve is actually faster than diagonal-preconditioned CG.
We also note that in four of the six $p=4$ tests in 3D, the fixed-point iteration did not converge in 5000 iterations with $\Mdiag^{-1}$, while $\Mbras^{-1}$ yielded convergence on the order of 100 iterations.

\begin{figure}[b!]
    \centering
    \includegraphics[width=1.0\linewidth]{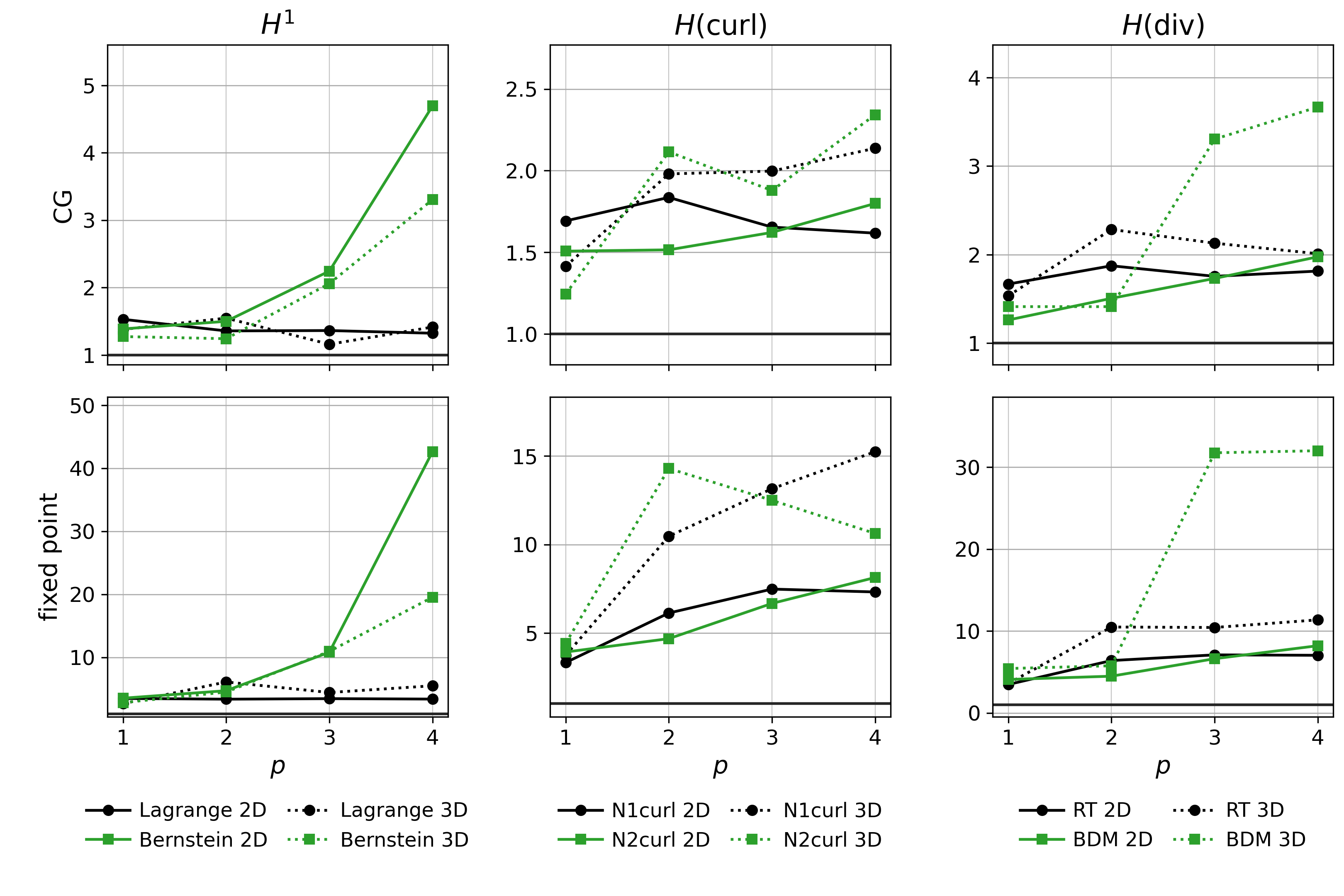}
    \caption{
Summary of the mass-solve timings reported across \cref{tab:h1_mass_solve,tab:hcurl_mass_solve,tab:hdiv_mass_solve}.
The top row reports the CG solve-time ratio \(t_{\rm diag}/t_{\rm bras}\) for polynomial
degrees \(p=1,2,3,4\), and the bottom row reports the analog for the fixed-point solve times.  
The columns correspond to \(H^1\),
\(H(\operatorname{curl})\), and \(H(\operatorname{div})\), respectively, and include
both finite element families and both spatial dimensions shown in the tables. Values larger than one indicate the associated CG or fixed-point solve time for BRAS is faster than that of diagonal preconditioning.
}
\label{fig:mass-speedups}
\end{figure}

\cref{fig:mass-speedups} shows the speedup provided by BRAS by plotting the solve-time ratio \(t_{\rm diag}/t_{\rm bras}\) for the CG and fixed-point runs in each table. In all cases considered, BRAS yields faster CG and fixed-point solve times than diagonal preconditioning, starting at 30--50\% improvements for $p=1$, and generally increasing with \(p\), significantly in some cases. We again point out that BRAS improvements in the fixed-point setting are much stronger than those for CG.

Wathen and Rees \cite{wathen2008chebyshev} observed that, for diagonally preconditioned low-order
\(H^1\) mass matrices, Chebyshev iteration gives convergence very
similar to CG when accurate spectral bounds are available
\cite{wathen2008chebyshev}.  
We observe the same behavior in the
broader setting considered here, including for both diagonal and BRAS
preconditioning.  In particular, across
essentially all of our tests, the observed CG convergence rates are remarkably close to the standard worst-case estimate
\begin{equation*}
    \frac{\|\bm{e}_k\|_{A_{{\wt M}}}}
         {\|\bm{e}_0\|_{A_{{\wt M}}}}
    \leq
    2
    \left(
        \frac{\sqrt{\wh{\kappa}}-1}{\sqrt{\wh{\kappa}}+1}
    \right)^k,
\end{equation*}
with \(A_{{\wt M}}={\wt M}^{-1/2}M{\wt M}^{-1/2}\).
This is precisely the Chebyshev upper bound obtained from the spectral
interval of the preconditioned operator. This makes
Chebyshev iteration an attractive alternative to CG in practice, because it is linear, and it avoids the global dot products required by CG, which can
be advantageous on parallel machines.  To our knowledge, this observation has not previously been emphasized for high-order \(H^1\), \(H(\curl)\), and \(H(\div)\) mass matrices. 

Next we recall the theory of \cref{sec:piola}, which considered mass conditioning under mesh deformation.
To provide numerical support for \cref{cor:bras_piola_metric_contrast} therein, we
consider preconditioning an \(H(\div)\) mass on a mesh obtained by rotating and stretching the unit-square simplicial mesh. 
The Jacobian determinant in this example satisfies \(\det (J_K)=1\) on every cell, so that it isolates the effect of the tensor-valued metric weight appearing
in the Piola pullback for the \(H(\div)\) mass matrix.  Given a stretch
factor \(r\geq 1\), we consider the coordinate map
\begin{align} \label{eq:rot_stretch_map}
F_r(\bm{x})
=
\bm{c}
+
R_{\pi/4}
\begin{pmatrix}
r & 0 \\
0 & r^{-1}
\end{pmatrix}
R_{\pi/4}^{\top}
(\bm{x}-\bm{c}),
\qquad
\bm{c}=(1/2,1/2),
\end{align}
where \(R_{\pi/4}\) denotes rotation by \(\pi/4\).

The Jacobian for this affine map is \(J_K=DF_r\), with \(\det (J_K)=1\), and the singular values of
\(J_K\) are \(r\) and \(r^{-1}\).  Therefore the \(H(\div)\) Piola metric
\(G_K=J_K^\top J_K/\det (J_K)\) has eigenvalues \(r^2\) and \(r^{-2}\), so the metric
contrast appearing in \cref{cor:bras_piola_metric_contrast} satisfies
\[
    \frac{\beta}{\alpha}
    \leq
    \frac{\sup_{K\in\Th}\lambda_{\max}(J_K^\top J_K/\det (J_K))}
             {\inf_{K\in\Th}\lambda_{\min}(J_K^\top J_K/\det (J_K))}
    =
    \frac{\lambda_{\max}(J_K^\top J_K)}
             {\lambda_{\min}(J_K^\top J_K)}
    =
    r^4 .
\]
Thus, the values \(r=1,3,5,7\) measured in the experiment below correspond to metric contrasts of \(1,81,625,\) and \(2401\), respectively.

\begin{figure}[t!]
    \centering
    \begin{minipage}[c]{0.54\linewidth}
        \centering
        \includegraphics[width=\linewidth]{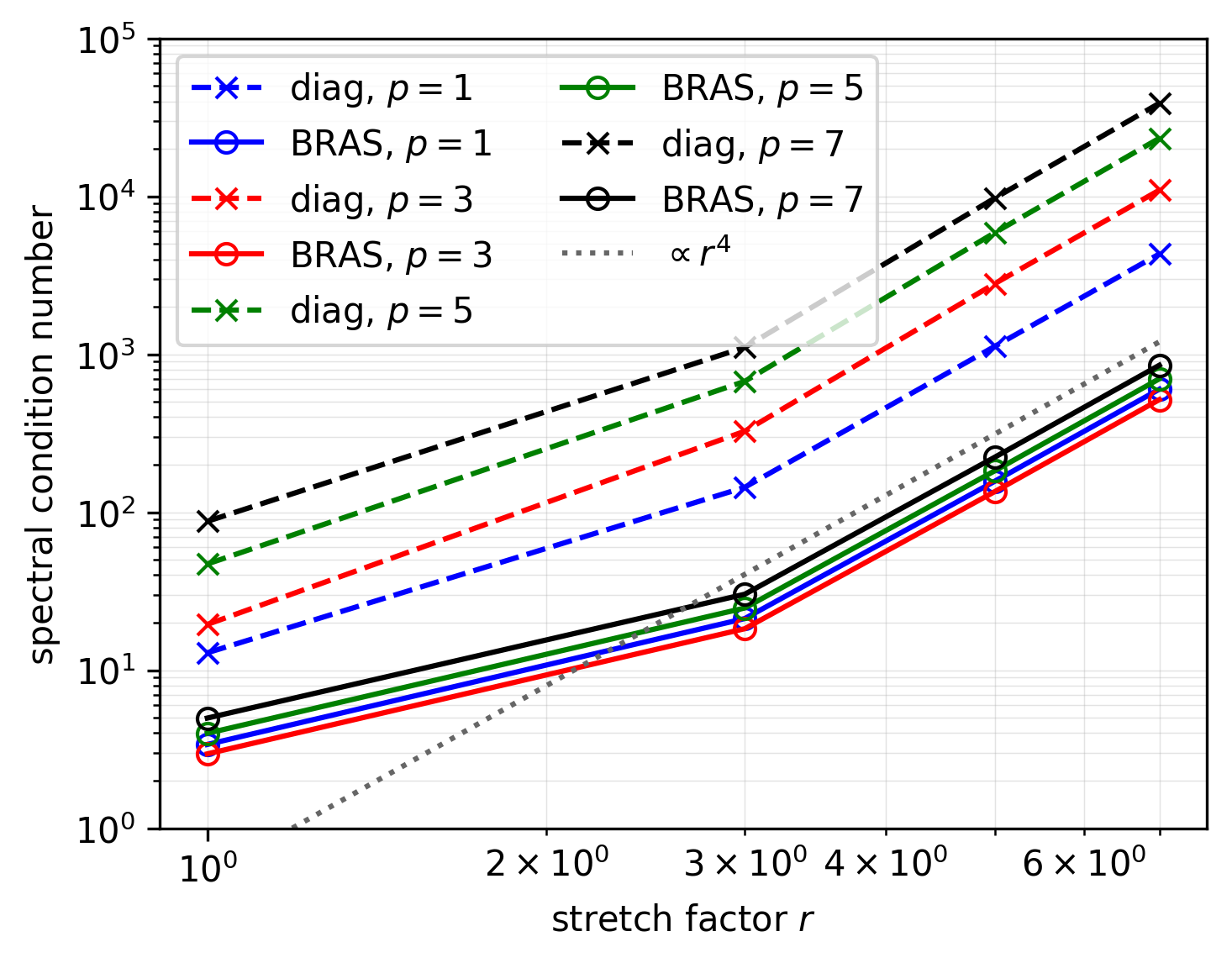}
    \end{minipage}
    \hfill
    \begin{minipage}[c]{0.32\linewidth}
        \centering
        \includegraphics[width=0.82\linewidth]{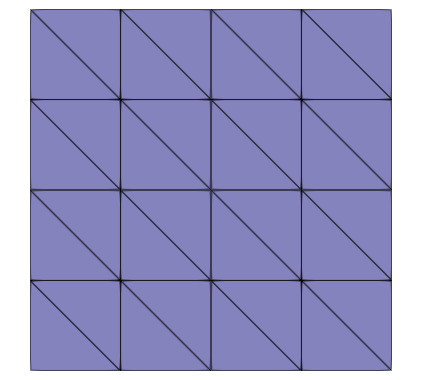}

        \vspace{0.04\linewidth}

        \includegraphics[width=0.82\linewidth]{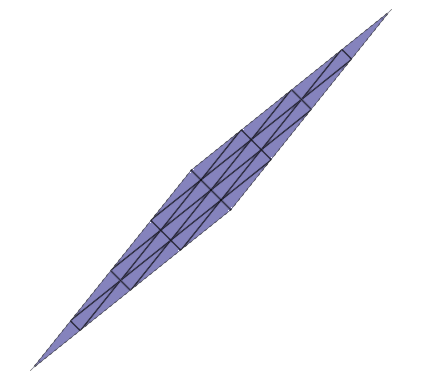}
    \end{minipage}

    \caption{Effect of the rotated-stretch map
\eqref{eq:rot_stretch_map} on the conditioning of the \(H(\div)\) BDM mass
matrix in 2D. The left panel shows spectral condition number
\(\wh{\kappa}(M,\Mdiag)\) and \(\wh{\kappa}(M,\Mbras)\) as functions of the
stretch parameter \(r\) for polynomial degrees \(p\in\{1,3,5,7\}\).
The right panels show representative
meshes for \(r=1\) and \(r=3\).  For ease of visualization, the pictured
meshes have fewer elements than those used in the experiments, which are
chosen separately for each \(p\) so that the corresponding generalized
eigenvalue problems have approximately \(5\times 10^3\) conforming DOFs.  The
mesh panels are shown with independent axis limits to emphasize element shape
rather than physical extent.}
    \label{fig:rotated-stretch-distortion}
\end{figure}

The general estimate in \cref{cor:bras_piola_metric_contrast} is stated for
\(\lambda_{\max}\), not directly for the spectral condition number.  However, in this
particular example the minimum eigenvalue satisfies
\(\lambda_{\min}(\Mbras^{-1}M)=1\) on both the reference and stretched meshes, for the reasons discussed after \cref{cor:bras_lambda_min}. Hence the spectral condition number bound stated in \cref{cor:bras_piola_metric_contrast} also applies in this case.
The resulting spectral condition numbers are shown in
\Cref{fig:rotated-stretch-distortion} for polynomial degrees
\(p=1,3,5,7\).  The BRAS spectral condition numbers grow essentially like \(r^4\), consistent with the metric-contrast factor \(\beta / \alpha \leq r^4\) above; thus, BRAS does not remove the Piola metric dependence.  However, at each fixed value of \(r\), the BRAS
spectral condition numbers are substantially smaller than their diagonal counterparts,
and the BRAS curves remain nearly flat as the polynomial degree is increased.
By contrast, the diagonal spectral condition numbers deteriorate substantially with
\(p\), roughly doubling with each increment of \(p\) shown.

\section{Block preconditioning}
\label{sec:block_preconditioning}

We now show why the quality of a mass inverse approximation matters when the inverse appears inside a Schur-complement preconditioner. Consider the
nonsingular symmetric saddle-point matrix
\begin{equation}
\label{eq:saddle_matrix_general_22}
    \mathcal A =
    \begin{bmatrix}
        M & G^{\top} \\
        G & -C
    \end{bmatrix},
\end{equation}
where $C$ is SPSD and \(M\) is an SPD mass matrix. 
This type of matrix is commonplace throughout scientific computing \cite{benzi2005saddle}.
Eliminating the first block gives the positive
Schur complement
\begin{equation}
\label{eq:positive_schur_general}
    -S = C + GM^{-1}G^{\top} \succ 0 .
\end{equation}
The Schur complement is the central object in block preconditioning for
\eqref{eq:saddle_matrix_general_22}.  For standard block diagonal and block
triangular preconditioners, the convergence of the resulting fixed-point or
Krylov iteration is governed by the conditioning and spectral distribution of
the preconditioned Schur complement, \({\wt S}^{-1}S\), for Schur complement approximation \({\wt S}\approx S\) \cite{benzi2005saddle,notay2014new,southworth2020fixed}.

In the setting considered here, the Schur complement approximation is obtained by replacing
the exact mass inverse in \eqref{eq:positive_schur_general} by a sparse SPD
approximation \(\Mwtinv\):
\begin{equation}
\label{eq:approx_positive_schur_general}
    -{\wt S} \coloneqq C + G\Mwtinv G^{\top} \succ 0 .
\end{equation}
Below we show that spectral bounds for the preconditioned mass operator \(\Mwtinv M\) directly imply spectral bounds for \({\wt S}^{-1}S\). Thus, a stronger approximation to \(M^{-1}\) may give a better conditioned preconditioned Schur complement, which in turn improves the convergence of the outer iteration on the saddle-point system. Note that the same question of mass inverse approximation arises in nonsymmetric mixed formulations and saddle-point problems as well; here we assume symmetry to enable the spectral analysis in the following section.

\subsection{Spectral mass bounds imply spectral Schur complement bounds}
\label{subsec:schur_theory}

For approximate mass inverse $\Mwtinv$, let
\(
\alpha_C(\Mwtinv) = \lambda_{\min}( S_C^{-1} {\wt S}_C )
\)
and
\(
\beta_C(\Mwtinv) = \lambda_{\max}( S_C^{-1} {\wt S}_C )
\)
denote the constants of spectral equivalence as in \eqref{eq:spec-equiv} for $({\wt S}_C, S_C)$, so that $\wh{\kappa}({\wt S}_C, S_C) = \wh{\kappa}(S_C,{\wt S}_C) = \frac{\beta_C(\Mwtinv)}{\alpha_C(\Mwtinv)}$. 
Further define related constants 
\begin{equation}
\label{eq:schur_a0-c0}
\begin{aligned}
    \alpha_0(\Mwtinv)
    :=
    \min_{\bm q \neq \bm 0}
    \frac{\bm q^{\top}G\Mwtinv G^{\top}\bm q}{\bm q^{\top}GM^{-1} G^{\top}\bm q},
    \qquad
    \beta_0(\Mwtinv)
    :=
    \max_{\bm q\neq \bm 0}
    \frac{\bm q^{\top}G\Mwtinv G^{\top}\bm q}{\bm q^{\top}GM^{-1} G^{\top}\bm q},
\end{aligned}
\end{equation}
which correspond to \(\alpha_C(\Mwtinv)\) and \(\beta_C(\Mwtinv)\) when \(C=0\), as well as a spectral equivalence of $(\widetilde M, M)$ restricted to the subspace defined by $\range(G^\top)$. In the following results we show that \(\alpha_C\) and \(\beta_C\) are controlled by \(\alpha_0\) and \(\beta_0\), and that those in turn are controlled by the spectrum of the preconditioned mass matrix.

\begin{lemma}
\label{lem:schur_condition_number_from_mass}
Let \(M\) and \(\Mwtinv\) be SPD, let \(C\succeq 0\), and assume
\(-S_C\succ 0\) and \(-{\wt S}_C\succ 0\). Then
\begin{equation}
\label{eq:schur_condition_number_general_bound}
    \wh{\kappa}(S_C,{\wt S}_C)
    \le
    \frac{\max\{1,\beta_0(\Mwtinv)\}}
         {\min\{1,\alpha_0(\Mwtinv)\}}
    \le
    \frac{\max\{1,\lambda_{\max}(\Mwtinv M)\}}
         {\min\{1,\lambda_{\min}(\Mwtinv M)\}} .
\end{equation}
If $C=0$,
\begin{equation}
\label{eq:schur_condition_number_zero_C_bound}
    \wh{\kappa}(S_0, {\wt S}_0)
    =
    \frac{\beta_0(\Mwtinv)}{\alpha_0(\Mwtinv)}
    \le
    \wh{\kappa}(M,\widetilde{M}).
\end{equation}
\end{lemma}

\begin{proof}
Let \(\bm q\neq \bm 0\), and write
\[
    c=\bm q^{\top}C\bm q,\qquad
    m=\bm q^{\top}GM^{-1}G^{\top}\bm q,\qquad
    {\wt m}=\bm q^{\top}G\Mwtinv G^{\top}\bm q .
\]
Then \(c\ge0\), \(m > 0\), and \(c+m>0\). Writing
\(\bm g=G^{\top}\bm q\),
\[
    \frac{{\wt m}}{m}
    =
    \frac{\bm g^{\top}\Mwtinv \bm g}
         {\bm g^{\top}M^{-1}\bm g}
    \in
    [\alpha_0(\Mwtinv),\beta_0(\Mwtinv)] .
\]
Now consider that
\[
    \frac{\bm q^{\top}{\wt S}_C\bm q}{\bm q^{\top}S_C\bm q}
    =
    \frac{c+{\wt m}}{c+m}
    =
    \frac{c}{c+m}
    +
    \frac{m}{c+m}\frac{{\wt m}}{m}    
    =
    \frac{c}{c+m}
    +
    \left(1 - \frac{c}{c+m} \right)\frac{{\wt m}}{m}    
\]
is a convex combination of the set \(\{1,\, \frac{{\wt m}}{m} \}\), and hence is bounded by the extrema of its elements,
\[
 \frac{\bm q^{\top}{\wt S}_C\bm q}{\bm q^{\top}S_C\bm q}
    \in
    [\min\{1,\alpha_0(\Mwtinv)\},
     \max\{1,\beta_0(\Mwtinv)\}]. 
 \]
Therefore
\begin{align*}
    \alpha_C(\Mwtinv)
    &=
    \lambda_{\min}( S_C^{-1} {\wt S}_C )
    \ge
    \min\{1,\alpha_0(\Mwtinv)\},\\
    \beta_C(\Mwtinv)
    &=
    \lambda_{\max}( S_C^{-1} {\wt S}_C )
    \le
    \max\{1,\beta_0(\Mwtinv)\}.
\end{align*}
Now note that \(\alpha_0(\Mwtinv)\) and \(\beta_0(\Mwtinv)\) are the extremal
Rayleigh quotients of \(\Mwtinv\) relative to \(M^{-1}\), restricted to
\(\range(G^{\top})\). Hence
\[
    \lambda_{\min}(\Mwtinv M)
    \le
    \alpha_0(\Mwtinv)
    \le
    \beta_0(\Mwtinv)
    \le
    \lambda_{\max}(\Mwtinv M),
\]
Combining, it follows that 
\[
    \wh{\kappa}(S_C,{\wt S}_C)
    =
    \frac{\beta_C(\Mwtinv)}{\alpha_C(\Mwtinv)}
    \le
    \frac{\max\{1,\beta_0(\Mwtinv)\}}
         {\min\{1,\alpha_0(\Mwtinv)\}}
             \le
    \frac{\max\{1,\lambda_{\max}(\Mwtinv M)\}}
         {\min\{1,\lambda_{\min}(\Mwtinv M)\}}.
\]

If \(C=0\), then the preconditioned Schur complement Rayleigh quotient is exactly the restricted
mass-inverse quotient:
\[
    \frac{\bm q^{\top}{\wt S}_0\bm q}{\bm q^{\top}S_0\bm q}
    =
    \frac{\bm g^{\top}\Mwtinv \bm g}
         {\bm g^{\top}M^{-1}\bm g},
    \qquad
    \bm g=G^{\top}\bm q .
\]
Since \(-S_0\) is SPD, \(G^{\top}\bm q\neq \bm 0\) for every
\(\bm q\neq \bm 0\). Thus the sharp constants for
\(({\wt S}_0,S_0)\) are \(\alpha_0(\Mwtinv)\) and
\(\beta_0(\Mwtinv)\), and therefore
\[
    \wh{\kappa}(S_0, {\wt S}_0)
    =
    \frac{\beta_0(\Mwtinv)}{\alpha_0(\Mwtinv)}
    \le
    \frac{\lambda_{\max}(\Mwtinv M)}
         {\lambda_{\min}(\Mwtinv M)}
    =
    \wh{\kappa}(M,\widetilde{M}).
\]
\end{proof}

\begin{corollary}
\label{cor:bras_schur_condition_number}
Let
\(
    -{\wt S}_{C,\mathrm{bras}}
    =
    C+G\Mbras^{-1}G^{\top}.
\)
Then
\begin{equation}
\label{eq:bras_schur_condition_number_C_bound}
    \wh{\kappa}(S_C,{\wt S}_{C,\mathrm{bras}})
    \le
    \beta_C(\Mbras^{-1})
    \le
    \beta_0(\Mbras^{-1})
    \le
    \lambda_{\max}(\Mbras^{-1}M).
\end{equation}
\end{corollary}

\begin{proof}

The result therefore follows from \cref{lem:schur_condition_number_from_mass}
with \(\Mwtinv=\Mbras^{-1}\) combined with  \cref{cor:bras_lambda_min}, which showed that
\(
    \lambda_{\min}(\Mbras^{-1}M) \geq 1.
\)
Specifically, plugging \(\Mwtinv=\Mbras^{-1}\) into the proof above, we have 
\(
\alpha_C(\Mbras^{-1}) \geq \min \{ 1, \alpha_0(\Mbras^{-1}) \geq \min \{ 1, \lambda_{\min}(\Mbras^{-1}M) \} \geq 1
\).
\end{proof}

Note in the case \(C=0\), the bound  \(\wh{\kappa}(S_0,{\wt S}_{0,\mathrm{bras}})
    \le
    \wh{\kappa}(M,\Mbras)\) can be strict:
the Schur-complement Rayleigh quotients $\alpha_0, \beta_0$ only sample the preconditioned mass \(\Mbras^{-1}M\) on \(\range(G^{\top})\). Thus, if either of the extremal modes of \(\Mbras^{-1}M\)
lie outside this range we get the strict inequality \(\wh{\kappa}(S_0,{\wt S}_{0,\mathrm{bras}}) < \wh{\kappa}(M,\Mbras)\).
In fact, experiments in \cref{sec:mixed_poisson} show that \(\range(G^{\top})\) may very effectively filter out the extremal modes of \(\Mbras^{-1}M\), where in some cases we observe \(\wh{\kappa}(S_0,{\wt S}_{0,\mathrm{bras}}) \ll \wh{\kappa}(M,\Mbras)\).

\section{Numerical results: Block preconditioning}
\label{sec:block_preconditioning_numerics}

This section tests Schur-complement-based block preconditioners for saddle-point and nonsymmetric mixed systems, wherein we use either a diagonal or BRAS mass inverse approximation.  
The approximate Schur complements are inverted only approximately using algebraic multigrid (AMG) inside an outer Krylov method.
Note that diagonal and BRAS inverses generally lead to different approximate Schur complement matrices, with the BRAS-based matrices being relatively denser. In the
experiments below we report results for small polynomial degrees \(p\), where the AMG methods considered here give reliable Schur-complement preconditioners for both the diagonal and BRAS choices.  Developing \(p\)-robust solvers for these
approximate Schur complements is a separate issue. We note that if such \(p\)-robust solvers were
available, the theory of \Cref{subsec:schur_theory}, together with the mass-only
results of \Cref{sec:mass_numerics}, suggests that the gains from BRAS over diagonal preconditioning would increase with \(p\).
We first consider a mixed Poisson problem
in \cref{sec:mixed_poisson} and then a mixed biharmonic problem in \cref{sec:mixed_biharmonic}.

\subsection{Mixed Poisson in \(H(\operatorname{div}) \times L^2\)}
\label{sec:mixed_poisson}

Consider the Poisson problem\\
\(
    -\div \grad u = f
\)
on \(\Omega=(0,1)^d\), \(d=2,3\), with homogeneous Dirichlet data for \(u\) on
\(\partial\Omega\).  We write this equation in first-order form as
\begin{equation}
\label{eq:poisson_first_order}
    \bm{\sigma} - \grad u = 0 \quad \mathrm{in}\; \Omega,
    \qquad
    \operatorname{div}\bm{\sigma} = -f \quad  \mathrm{in}\; \Omega,
    \qquad
    u = 0 \quad \mathrm{on}\;  \partial \Omega.
\end{equation}
The mixed weak problem is: find
\((\bm{\sigma},u)\in H(\operatorname{div};\Omega)\times L^2(\Omega)\) such that
\begin{subequations}
\label{eq:mixed_poisson_weak}
\begin{align}
    (\bm{\sigma},\bm{\tau})_\Omega
    +
    (u,\operatorname{div}\bm{\tau})_\Omega
    &=
    0
    &&\forall \bm{\tau}\in H(\operatorname{div};\Omega),
    \\
    (\operatorname{div}\bm{\sigma},v)_\Omega
    &=
    -(f,v)_\Omega
    &&\forall v\in L^2(\Omega).
\end{align}
\end{subequations}
We use the \(\mathrm{BDM}_p/\mathrm{DG}_{p-1}\) mixed finite element pair, denoted by \(\Sigma_h\subset H(\operatorname{div};\Omega)\) and \(Q_h\subset L^2(\Omega)\), respectively; see
\cite[Section~7.1.2]{boffi2013mixed}.
The discrete saddle-point problem is
\begin{equation}
\label{eq:mixed_poisson_saddle}
    \begin{bmatrix}
        M & B^{\top} \\
        B & 0
    \end{bmatrix}
    \begin{bmatrix}
        \bm{\sigma} \\
        \bm{u}
    \end{bmatrix}
    =
    \begin{bmatrix}
        \bm 0 \\
        - \bm f
    \end{bmatrix},
\end{equation}
where \(M_{ij} = (\bm{\phi}_j,\bm{\phi}_i)_\Omega\),
\(B_{ij} = (\operatorname{div}\bm{\phi}_j,\chi_i)_\Omega\), and \((\bm f)_i = ( f, \chi_i )_{\Omega}\). 
Eliminating
\(\bm{\sigma}\) gives the positive \(u\)-block Schur complement
\(-S=BM^{-1}B^{\top} \succ 0\), which is generically dense since $M^{-1}$ is generically dense. 

We consider simplicial meshes, consisting of triangles in two dimensions and tetrahedra in three. 
The two discrete spaces used here have different sizes, depending on $p$ and the spatial dimension; ignoring
boundary effects, the ratio between the number of \(\Sigma_h\)-DOFs and
\(Q_h\)-DOFs is approximately
\[
\begin{aligned}
d=2:\;
\frac{\dim\Sigma_h}{\dim Q_h}
&\approx
\begin{cases}
3, & p=1,\\
2.5, & p=2,
\end{cases}
&
\qquad
d=3:\;
\frac{\dim\Sigma_h}{\dim Q_h}
&\approx
\begin{cases}
6, & p=1,\\
4.5, & p=2.
\end{cases}
\end{aligned}
\]
Thus, the \(\bm{\sigma}\)-block represents a substantial fraction of the block system \eqref{eq:mixed_poisson_saddle},
especially at low order, which is an important observation for the preconditioning strategy we describe below.

We solve \eqref{eq:mixed_poisson_saddle} with MINRES \cite{paige1975minres} using a block-diagonal preconditioner \(\mathcal P_{X_{\bm \sigma},X_u}^{-1}\) that approximates the leading mass block and Schur complement as follows:\footnote{We have also considered solving the saddle-point system with FGMRES using a block-triangular
preconditioner, including with CG-accelerated applications of both block solves in the triangular preconditioner.  
The results were qualitatively similar to the MINRES results reported here, though with slightly better performance of BRAS relative to diagonal preconditioning than in the MINRES results here.  
For simplicity, we only present the MINRES results.}
\begin{equation}
\label{eq:mixed_poisson_split_minres_prec}
    \mathcal P_{X_{\bm \sigma},X_u}^{-1}
    :=
    \begin{bmatrix}
        X_{\bm \sigma} & 0 \\
        0 & \mathcal V_{X_u}
    \end{bmatrix}
    \approx
    \begin{bmatrix}
        M^{-1} & 0 \\
        0 & -S^{-1}
    \end{bmatrix}
    =:
    \mathcal P_{\rm ideal}^{-1}.
\end{equation}
Here, \(X_{\bm \sigma},X_u\in\{\Mdiag^{-1},\Mbras^{-1}\}\), and
\(\mathcal V_{X_u}\) denotes one AMG V-cycle for the approximate Schur complement
\(-{\wt S}_{X_u}=BX_uB^{\top}\). Note that there is an additional indirect effect on convergence related to the efficacy of AMG on the approximate Schur complement formed with $\Mdiag^{-1}$ vs. $\Mbras^{-1}$, but we do not consider this comparison in detail.

\cref{tab:mixed-poisson-mass-schur-spectra} presents spectral condition numbers of the associated preconditioned operators on some fixed, small meshes.
The quantity \(\wh{\kappa}( M, X_{\bm \sigma}^{-1})\) measures the mass approximation
that appears in the \(\bm{\sigma}\)-block, while
\(\wh{\kappa}(S, {\wt S}_{X_u})\) measures the Schur complement approximation assuming exact inversion of \({\wt S}_{X_u}\), rather than one AMG V-cycle.  
The spectral condition number \(\wh{\kappa}( S, {\wt S}_{X_u})\) is covered by the \(C=0\) special case of \Cref{lem:schur_condition_number_from_mass}; rewriting that result to the current notation, we have:
\begin{equation}
\label{eq:schur_condition_number_zero_C_bound_copy}
    \wh{\kappa}( S,{\wt S}_{X_u})
    =
    \frac{
    \displaystyle
    \max_{\bm q\neq \bm 0}
    \frac{\bm q^{\top}B X_u B^{\top}\bm q}{\bm q^{\top}B M^{-1} B^{\top}\bm q}
    }{
    \displaystyle
    \min_{\bm q \neq \bm 0}
    \frac{\bm q^{\top}B X_u B^{\top}\bm q}{\bm q^{\top}B M^{-1} B^{\top}\bm q}
    }
    \leq
    \frac{
    \displaystyle
    \max_{\bm g\neq \bm 0}
    \frac{\bm g^{\top} X_u\bm g}{\bm g^{\top} M^{-1}\bm g}
    }{
    \displaystyle
    \min_{\bm g \neq \bm 0}
    \frac{\bm g^{\top} X_u \bm g}{\bm g^{\top} M^{-1} \bm g}
    }
    =
    \wh{\kappa}( M, X_u^{-1}).
\end{equation}
That is, the spectral condition number of the preconditioned Schur complement is bounded by that of the preconditioned mass, and more
sharply depends only on the action of the preconditioned mass \(X_u M\) on \(\operatorname{range}(B^{\top})\).
This restricted nature is clearly visible in \cref{tab:mixed-poisson-mass-schur-spectra}, where spectral condition numbers of preconditioned Schur complements are typically much smaller than the corresponding preconditioned mass spectral condition numbers \(\wh{\kappa}( S,{\wt S}_{X_u}) \ll \wh{\kappa}( M, X_u^{-1})\), particularly in three dimensions and at higher order.
While BRAS gives a uniformly better approximation to
\(M^{-1}\) than diagonal scaling in the full mass comparison, i.e., \(\wh{\kappa}( M, \Mbras) \ll \wh{\kappa}( M, \Mdiag)\), its improvement in the Schur complement approximation is more nuanced. %
In the precision displayed, for $p=1$ the diagonal/BRAS values are
\(2.9/2.9\) in two dimensions and \(3.0/3.0\) in three dimensions.  Thus, at
\(p=1\), the BRAS approximation improves the \(\bm{\sigma}\)-block but does not improve the Schur complement approximation in the $u$ block, despite the fact that the BRAS-based Schur complement is more expensive to build and to use in AMG.  This motivates the ``hybrid'' preconditioning choice of
\(X_{\bm \sigma}=\Mbras^{-1}\), \(X_u=\Mdiag^{-1}\).
Note that for \(p\ge2\) BRAS uniformly improves the condition number of the preconditioned Schur-complement.

\begin{table}[t!]
\centering
\begingroup
\setlength{\tabcolsep}{3pt}
\begin{tabular}{lcccccccc}
\toprule
& \multicolumn{4}{c}{Two dimensions} 
& \multicolumn{4}{c}{Three dimensions} \\
\cmidrule(lr){2-5} \cmidrule(lr){6-9}
& \(p=1\) & \(p=2\) & \(p=3\) & \(p=4\)
& \(p=1\) & \(p=2\) & \(p=3\) & \(p=4\) \\
\midrule
\(\wh{\kappa}(M,X_{\bm \sigma}^{-1})\)
& \(13/3.4\) & \(11/2.7\) & \(20/3.0\) & \(31/3.4\)
& \(39/7.1\) & \(32/5.4\) & \(210/5.7\) & \(310/6.5\) \\
\(\wh{\kappa}(S,{\wt S}_{X_u})\)
& \(2.9/2.9\) & \(5.3/1.8\) & \(8.1/1.8\) & \(9.9/2.1\)
& \(3.0/3.0\) & \(5.3/2.1\) & \(12/2.1\) & \(30/2.4\) \\
\bottomrule
\end{tabular}
\endgroup
\caption{Spectral condition number comparison of diagonal and BRAS mass-inverse approximations for the mixed Poisson problem. Entries are diagonal/BRAS. Here \(X\in\{\Mdiag^{-1},\Mbras^{-1}\}\), \(-S=BM^{-1}B^{\top}\), and \(-{\wt S}_X=BXB^{\top}\). The cases were chosen so that the discontinuous scalar space \(\Qh\) has approximately \(5\times 10^3\) DOFs.}
\label{tab:mixed-poisson-mass-schur-spectra}
\end{table}

Next we consider the numerical solution of \eqref{eq:mixed_poisson_saddle} with MINRES and the preconditioner 
    \(\mathcal P_{X_{\bm \sigma},X_u}^{-1}\) from \eqref{eq:mixed_poisson_split_minres_prec}.
    We use the MINRES implementation provided by SciPy \cite{virtanen2010scipy}, with tolerance \(\texttt{rtol=1e-10} \| \bm f \|\), and zero initial iterate, and the source function $f$ chosen randomlu.
We consider three preconditioning variations, given by ``diag'', corresponding to \(X_{\bm \sigma}=X_u=\Mdiag^{-1}\), ``BRAS'', corresponding to \(X_{\bm \sigma}=X_u=\Mbras^{-1}\), and ``hybrid'' corresponding to \(X_{\bm \sigma}=\Mbras^{-1}\), \(X_u=\Mdiag^{-1}\).
The AMG hierarchy for \({\wt S}_{X_u}\) is constructed using classical Ruge--St\"uben AMG from the PyAMG library \cite{bell2023pyamg}. We use classical strength of connection, RS coarsening with a second pass, classical interpolation, and symmetric Gauss--Seidel pre- and post-smoothing. 
The strength-of-connection threshold is selected from the fixed candidate set
\(\{0.15,0.25\}\).  For each fixed choice of dimension, $p$, and Schur-complement approximation \(X_u\), we choose the strength threshold to be the one that yields the fastest convergence of standalone AMG on \({\wt S}_{X_u}\) across a variety of representative meshes.

\begin{figure}[t!]
    \centering
    \includegraphics[width=0.95\linewidth]{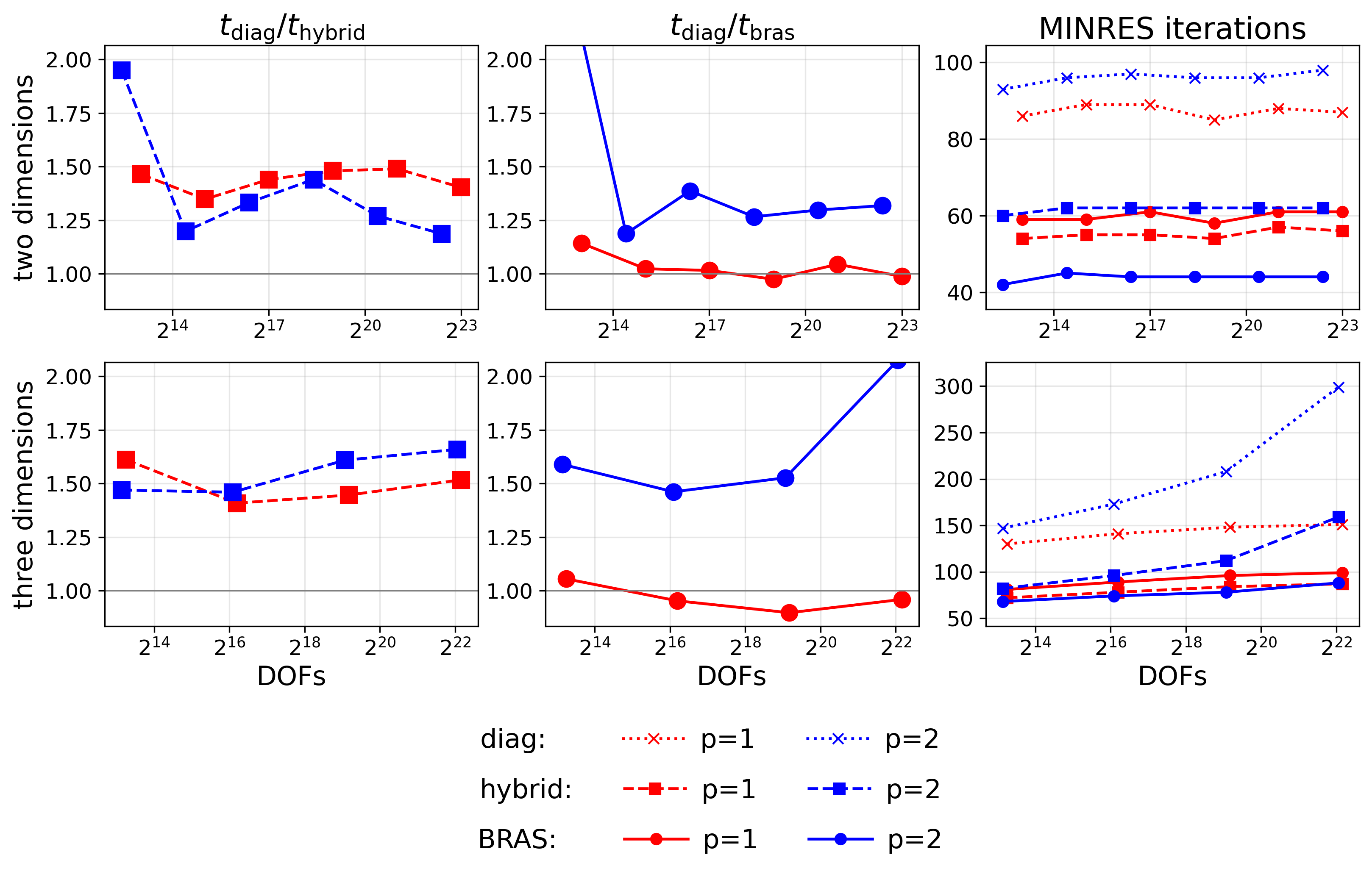}
\caption{
MINRES performance for the mixed Poisson problem with the
\(\mathrm{BDM}_p/\mathrm{DG}_{p-1}\) pair.  The top row shows two-dimensional
results and the bottom row shows three-dimensional results.  The left and middle
columns report solve-time ratios \(t_{\rm diag}/t_{\rm hybrid}\) and
\(t_{\rm diag}/t_{\rm bras}\), where ``hybrid'' denotes
\(X_{\bm \sigma}=\Mbras^{-1}\), \(X_u=\Mdiag^{-1}\).  Values above one therefore
indicate improvement over the diag preconditioner.  The right column
reports the corresponding MINRES iteration counts for the diag,
hybrid, and BRAS preconditioners.
}
\label{fig:mixed-poisson-minres-hybrid}
\end{figure}

\Cref{fig:mixed-poisson-minres-hybrid} compares the three preconditioning
approaches: diag, hybrid, and BRAS.  The left column reports
\(t_{\rm diag}/t_{\rm hybrid}\), the middle column reports
\(t_{\rm diag}/t_{\rm bras}\), and the right column reports the corresponding
MINRES iteration counts.  Thus values larger than one in the first two columns
indicate faster solves than diag.
For $p=1$, the hybrid preconditioner gives the strongest performance, reducing MINRES iteration counts and solve times relative to diag. BRAS reduces MINRES iterations over diag, but not solve time, which is consistent with
\Cref{tab:mixed-poisson-mass-schur-spectra}, because there BRAS improves the \(\bm{\sigma}\)-block mass approximation, but gives essentially the same conditioning for the Schur complement as diagonal. 
For \(p=2\), the hybrid preconditioner still improves substantially over diag, but BRAS gives the lowest iteration counts and
the best solve times on the largest cases, with 1.3--2$\times$ speedup over the diagonal reference.  This is consistent with the spectral condition number results from \cref{tab:mixed-poisson-mass-schur-spectra}, where at \(p=2\) BRAS improves both the \(\bm{\sigma}\)-block mass approximation and the $u$-block Schur-complement approximation.

\subsection{Mixed biharmonic in $H^1 \times H^1$}
\label{sec:mixed_biharmonic}

We next consider a two-dimensional, scalar mixed formulation arising from an anisotropic fourth-order tensor-divergence operator, which can be seen as a scalar version of the vector-valued electron viscosity term within Ohm's law in extended magnetohydrodynamics \cite{braginskii1965transport}. Discretizing the PDE problem in time with a diagonally implicit Runge-Kutta method results in a semi-discrete system of the form
\begin{align} \label{eq:biharmonic_semi}
\chi + \delta t \, \nabla^\perp \cdot \bm E = f, \qquad \bm E = -\nabla \cdot \big(\Pi_e(\bm j_e)\big), \qquad \bm j_e = \nabla^\perp \chi,
\end{align}
with constant ${\delta t > 0}$ the time-step size (possibly up to some proportionality constant).
Here, $\chi$ is an unknown scalar field (corresponding to a magnetic field component), $\bm E$ is the electric field, $\bm j_e$ is the current density, and $f$ is a known function. 
The electric field is equipped with perfect conducting wall boundary conditions $\bm{E} \cdot \bm{n}^\perp = 0$, and, for simplicity, we assume a zero traction condition ${{\Pi}_e(\bm{j}_e) \cdot \bm{n} = \bm 0}$, for boundary normal vector $\bm n$.
Finally, we set the anisotropic electron stress tensor to
\[
\Pi_e(\bm j_e) = -A \big(A \, : \, \nabla \bm j_e \big),
\qquad
A := \big(\bm b \otimes \bm b - \tfrac{1}{3}I\big)
\]
with \(\bm{b}\) a prescribed, unit-length anisotropy direction field.\footnote{In practice, both $\bm b$ and $\chi$ relate to the same magnetic field and are not independent of one another. Here, we consider them as separate fields, as would be the case e.g., in a Newton iteration step within a nonlinear implicit scheme for magnetohydrodynamics.}
One can show that \eqref{eq:biharmonic_semi} simplifies to the fourth-order equation
\[
    \chi+\delta t \, D^\ast_A D_A \chi=f,
    \qquad
    \mathrm{where}
    \:
    D_A u = A \, :\, \nabla(\nabla^\perp u).
\]
In our tests, we set $\delta t = h$, for spatial mesh size \(h\), in view of balancing spatial and temporal discretization errors; explicit time integration would require a CFL condition of the form $\delta t \lesssim h^4$, so the linear systems considered here are numerically stiff.

We now pose this problem in mixed form by introducing the auxiliary variable
\[
\zeta 
= 
\sqrt{\delta t}
A \, : \, \nabla \bm j_e 
= 
\sqrt{\delta t}
A
\, : \, \nabla (\nabla^\perp \chi),
\]
and the associated second-order coupling
\[
    c(v,u)
    =
    \sqrt{\delta t}
    \int_\Omega
        \nabla^\perp v\cdot \nabla\cdot\bigl(u\, A \bigr)
    \,\ud x .
\]
Given $V_h \subset H^{1}(\Omega)$, the mixed weak problem is to find \((\chi_h,\zeta_h)\in\Vh\times\Vh\) such that
\[
    (\eta_h,\chi_h)_\Omega-c(\eta_h,\zeta_h)
    =
    (\eta_h,f_h)_\Omega,
    \qquad
    (\gamma_h,\zeta_h)_\Omega+c(\chi_h,\gamma_h)
    =
    0
\]
for all \((\eta_h,\gamma_h)\in\Vh\times\Vh\).
Here \(V_h \subset H^1(\Omega)\) is an unrestricted finite-element space, and the aforementioned homogeneous boundary conditions are incorporated naturally in $c$ through integration by parts.
With
\(M_{ij}=(\phi_j,\phi_i)_\Omega\) and \(C_{ij}=c(\phi_i,\phi_j)\), this gives
\begin{align} \label{eq:biharmonic_2x2}
    \begin{bmatrix}
        M & C^{\top} \\
        -C & M
    \end{bmatrix}
    \begin{bmatrix}
        \bm{\zeta} \\
        \bm{\chi}
    \end{bmatrix}
    =
    \begin{bmatrix}
        \bm{0} \\
        \bm{f}
    \end{bmatrix}.
\end{align}
We solve this system with FGMRES \cite{saad1993flexible} using a lower triangular preconditioner,
\begin{align}
    \mathcal P_X
    =
    \begin{bmatrix}
        M & 0 \\
        -C & {\wt S}_X
    \end{bmatrix},
    \qquad
    {\wt S}_X := M+ C X C^{\top}
    \approx
    M+ C M^{-1}C^{\top}
    =: S.
\end{align}
For the mass inverse approximations, we consider \(X \in \{\Mdiag^{-1}, \, \Mbras^{-1} \}\), and we choose $\bm f$ to be a random vector. 
In terms of FGMRES, we initialize with the zero vector, iterate until \( \|\bm{r}_k\| \leq 10^{-10} \| \bm{r}_0 \| \), and we restart every 30 iterations.

We apply \({\cal P}_X^{-1}\) inexactly as follows. The mass inverse $M^{-1}$ in the $(1,1)$ block is approximated by CG preconditioned by $X$, iterated until $\texttt{rtol=1e-6}$ or the number of iterations exceeds 40.\footnote{This is a relatively tight tolerance, but this solve takes a trivial amount of the compute time relative to that of the $(2,2)$ block. This contrasts with the relative timing of the mass-only solve from the mixed-Poisson problem in the previous section, because here the number of DOFs in the (1,1) and (2,2) blocks is equal, and because the Schur complement is now a fourth-order operator.}
To approximate \({\wt S}_X^{-1}\) we use CG preconditioned by one AMG V-cycle, and iterate until $\texttt{rtol=1e-3}$ or the number of iterations exceeds 20.
Specifically, we use the least-squares algebraic-multigrid domain-decomposition (LS-AMG-DD) solver \cite{southworth2026lsamgdd}, which we have shown to be robust on a different fourth-order problem discretized directly in $H^2$ \cite{krzysik2026conforming}.
The LS-AMG-DD parameters used are as follows: 1 pass of aggregation for coarsening, pre- and post smoothing given by overlapping multiplicative Schwarz and its adjoint, respectively, the local spectral cutoff is $\tau_{\rm cut} = 1.5 \max_j {\rm mult}_{\omega}(j)$, the maximum number of levels in the hierarchy is capped at five, and the coarsening is halted if a coarse-level matrix density exceeds 25\%.

We remark that to apply LS-AMG-DD to a matrix ${\wt S}$ requires a Gram factor $G_{{\wt S}}$, such that ${\wt S} = G_{{\wt S}}^\top G_{{\wt S}}$ is a Gram matrix.
Given Gram factors $G_M$ and $G_X$ for the mass $M$ and its approximate inverse $X \approx M^{-1}$, such that $M = G_M^\top G_M$ and $X = G_X^\top G_X$, we write the approximate Schur complement as a Gram matrix as
\begin{align}
    {\wt S}_X = G_{{\wt S}}^{\top}G_{{\wt S}},
    \qquad
    G_{{\wt S}}
    =
    \begin{bmatrix}
        G_M \\
        G_X C^{\top}
    \end{bmatrix}.
\end{align}
Recalling that the conforming mass matrix can be written as $M = P^\top M_{\rm br} P$, we compute its Gram factor as $G_M = M_{\rm br}^{1/2} P$, using the approach outlined in \cite{krzysik2026conforming}.
The Gram factors for the approximate mass inverses are computed as follows: For \(X=\Mdiag^{-1}\), we take \(G_X=\Mdiag^{-1/2}\); for
\(X=\Mbras^{-1}\), we take \(G_X = M_{\rm br}^{-1/2}P D^{-1}\), according to \eqref{eq:mhatinvbras_def}.

\begin{figure}[b!]
    \centering
    \includegraphics[width=0.95\linewidth]{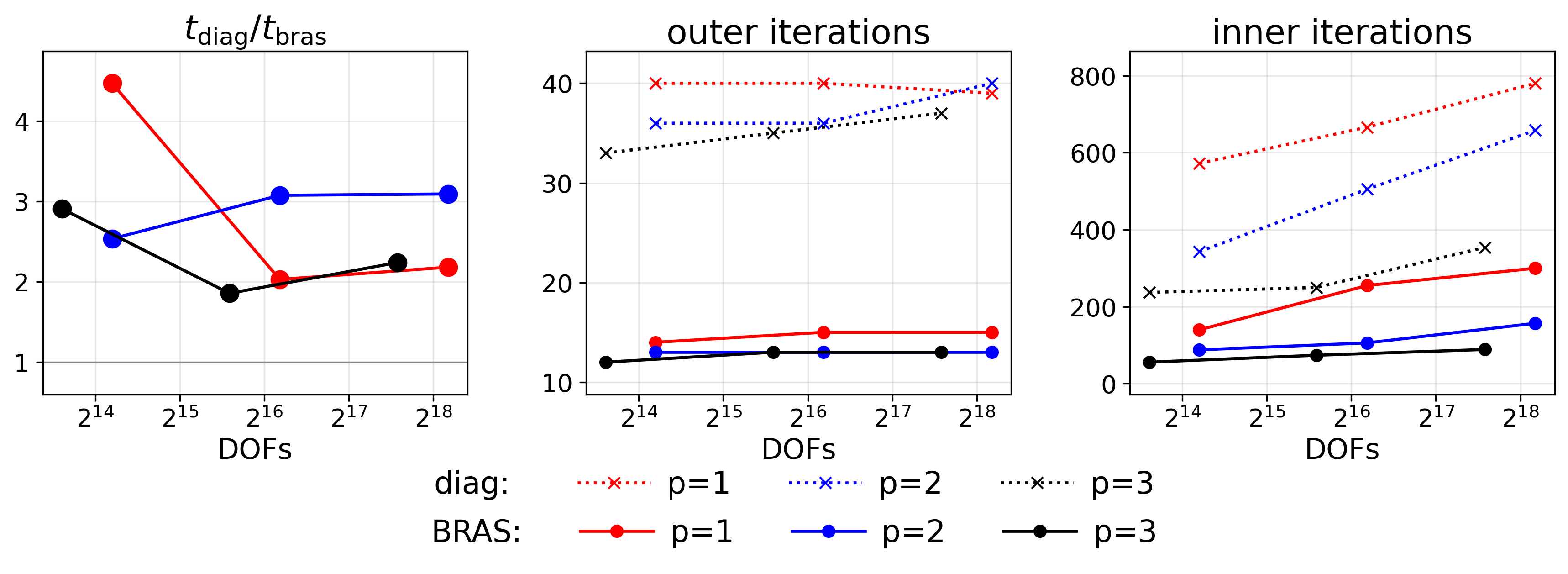}
    \caption{For the mixed biharmonic problem \eqref{eq:biharmonic_2x2}, the first column shows total FGMRES time using $X = M^{-1}_{\rm diag}$ relative to that using $X = M^{-1}_{\rm bras}$; a value larger than 1.0 corresponds to the BRAS approach being faster than the diagonal approach.
    The middle column shows total FGMRES iterations, and the third column shows total LS-AMG-DD iterations on the approximate Schur complement ${\wt S}_{X}$.
    }
    \label{fig:biharmonic}
\end{figure}

Numerical results are shown in \cref{fig:biharmonic} for Lagrange basis functions of degrees $p = 1,2,3$ using GLL nodes.
The domain \( \Omega = (0, 1)^2 \) is meshed with quadrilateral elements rather than triangular elements as in all other two-dimensional tests in the paper.
The direction vector is $\bm b = ( \cos \theta, \sin \theta )$ with $\theta = \pi/6 = 30^\circ$.
The first column shows the solve time of the diagonally-preconditioned approach relative to the BRAS-preconditioned approach; in all cases, BRAS yields a faster solve time than diagonal by a factor of $\sim 2$--$3 \times$.
Considering the middle column, the total number of FGMRES iterations is significantly decreased when using the BRAS strategy relative to the diagonal strategy, and the iteration counts are more tightly clustered as a function of $p$ for the BRAS approach.
Finally, considering the last column, the total number of LS-AMG-DD iterations is significantly smaller for the BRAS  approach, and also has much improved scaling with respect to the number of DOFs in the problem.
\section{Conclusions}
\label{sec:conclusions}

We introduce BRAS, a sparse algebraic approximate inverse for conforming
finite-element mass matrices.  BRAS is obtained by applying exact broken element mass inverses and averaging back to the conforming space.  The preconditioner is local, SPD, basis independent
at the algebraic level, and has the same element-adjacency sparsity as the
conforming mass matrix.  We characterize the preconditioned mass spectrum using a broken-space projector and show how mass-inverse approximation bounds pass to approximate Schur complements.  
The numerical results show that BRAS substantially reduces mass-matrix spectral condition numbers and Krylov iteration counts relative to diagonal scaling across \(H^1\), \(H(\operatorname{curl})\), and
\(H(\operatorname{div})\) spaces. For both explicit mass matrix solves and mixed block preconditioning problems, BRAS formulations lead to consistent solve-time improvements over diagonal preconditioning in essentially all of the tested cases, typically on the order of 1.5--3$\times$.

\section*{Acknowledgments}
The authors used OpenAI's ChatGPT during the preparation of this manuscript, including for exploratory mathematical discussion, coding, and drafting of the exposition.
All mathematical claims, proofs, computations, citations, and final wording were reviewed, verified, and edited by the authors, who take full responsibility for the manuscript.
Discussions with Steven Walton on aspects of mass preconditioning are gratefully acknowledged. 

\bibliographystyle{plain}
\bibliography{bras_refs.bib}

\end{document}